\documentclass[12pt,reqno]{amsart}
\usepackage{graphicx,stmaryrd,amsmath,amsfonts,latexsym,amssymb,amsthm,color, mathrsfs}
\usepackage{latexsym}
\usepackage{verbatim}
\usepackage{enumerate}
\usepackage{microtype}
\usepackage[utf8]{inputenc}
\usepackage[numbers]{natbib}

\usepackage{tikz}
\usetikzlibrary{matrix,backgrounds,arrows}

\usepackage[a4paper]{geometry}
\definecolor{blau}{rgb}{0.03,0.1,0.6}
\usepackage[pdftex]{hyperref} %links fuer arxiv entfernen
\hypersetup{
pdfauthor = {Francesco Bei, Batu Gueneysu, Joern Mueller},
pdftitle = {Scattering theory of the Hodge-Laplacian under a conformal perturbation},
pdfsubject ={Scattering theory}{Hodge Laplace}{conformal metrics}{},
colorlinks = true,
linkcolor = blau,
citecolor = blau
}

\newcommand{\dom}{\operatorname{Dom}}
\newcommand{\id}{\mathrm{id}}

\newcommand{\cf}{\psi}
\renewcommand{\gg}{{g_\cf}} 
\newcommand{\Q}{Q}
\newcommand{\Id}{\mathrm{d}}
\newcommand{\IL}{\mathsf{L}}
\newcommand{\ICC}{\mathsf{C}}
\newcommand{\IC}{\mathbb{C}}
\newcommand{\IR}{\mathbb{R}}
\newcommand{\IHH}{\mathscr{H}}
\newcommand{\IJJ}{\mathscr{J}}
\newcommand{\ILL}{\mathscr{L}}
\newcommand{\IDD}{\mathscr{D}}
\newcommand{\IMM}{\mathscr{M}}
\newcommand{\z}{\mu}
\newcommand{\IT}{\mathrm{T}}
\newcommand{\ITs}{{\mathrm{T}^*\!}}
\newcommand{\exte}{\mathrm{ext}}
\newcommand{\inti}{\mathrm{int}}

\newcommand{\IN}{\mathbb{N}}
\newcommand{\question}[1]
{\leavevmode{\marginpar{\tiny
$\hbox to 0mm{\hspace*{-0.5mm}$\leftarrow$\hss}
\vcenter{\vrule depth 0.1mm height 0.1mm width \the\marginparwidth}
\hbox to 0mm{\hss$\rightarrow$\hspace*{-0.5mm}}$\\\relax\raggedright #1}}}

\DeclareMathOperator{\supp}{supp}

\newcommand{\Reell}{\mathbb{R}}

\newcommand{\vol}{\mathrm{vol}}

\newcommand{\scr}[2]{{\left(#1,#2\right)}}
\newcommand{\sce}[2]{{\left\langle #1,#2\right\rangle}}

\newtheorem{theorem}{Theorem}[section]
\newtheorem{definition}[theorem]{Definition}

\newtheorem{lemma}[theorem]{Lemma}

\newtheorem{corollary}[theorem]{Corollary}
\newtheorem{proposition}[theorem]{Proposition}

\newtheorem{remark}[theorem]{Remark}

\DeclareMathOperator{\inj}{inj}
\newcommand{\myinf}[1]{\inf{}_{\!#1}}

\makeatletter
\def\yWedge{\mathchoice{\textstyle\bigwedge}{\bigwedge}{\bigwedge}{\bigwedge}}
\def\xWedge^#1{\yWedge^{\!\!#1}}
\def\Wedge{\@ifnextchar^{\xWedge}{\xWedge^{\,}}}
\makeatother
\newcommand{\slim}{\mathop{s\rule[0.5ex]{1ex}{0.1ex}\mathrm{lim}}}

\newenvironment{introenv}[2]{\medskip\textbf{#1~\ref{#2}.}\par\begin{it}}{\end{it}\medskip}

\title[Conformal Scattering]{Scattering theory of the Hodge-Laplacian under a conformal perturbation}

\author[F.~Bei]{Francesco Bei}
\address{Humboldt-Universit\"at zu Berlin,
Rudower Chaussee 25,
Institut f\"ur Mathematik,
12489 Berlin, Germany} \email{bei@math.hu-berlin.de}

\author[B.~G\"uneysu]{Batu G\"uneysu}
\address{Humboldt-Universit\"at zu Berlin,
Rudower Chaussee 25,
Institut f\"ur Mathematik,
12489 Berlin, Germany} \email{gueneysu@math.hu-berlin.de}

\author[J.~M\"uller]{J\"orn M\"uller}
\address{Humboldt-Universit\"at zu Berlin,
Rudower Chaussee 25,
Institut f\"ur Mathematik,
12489 Berlin, Germany} \email{jmueller@math.hu-berlin.de}

\begin{document}
%~\hfill{\today}

\begin{abstract} Let $g$ and $\tilde{g}$ be Riemannian metrics on a noncompact manifold $M$, which are conformally equivalent. We show that under a very mild \emph{first order} control on the conformal factor, the wave operators corresponding to the Hodge-Laplacians $\Delta_g$ and $\Delta_{\tilde{g}}$ acting on differential forms exist and are complete. We apply this result to Riemannian manifolds with a bounded geometry and more specifically, to warped product Riemannian manifolds with a bounded geometry. Finally, we combine our results with some explicit calculations by Antoci to determine the absolutely continuous spectrum of the Hodge-Laplacian on $j$-forms for a large class of warped product metrics.
\end{abstract}

\maketitle
\section*{Introduction}
\noindent One of the most fundamental problems in geometry is the determination of the spectrum of the Laplace operator corresponding to a Riemannian manifold $(M,g)$. Here, one is particularly interested in the Hodge-Laplace operator $\Delta^{(j)}_g$ which acts on differential $j$-forms, as the latter is directly linked to the de Rham complex, thus the topology of $M$. If $M$ is compact, then the spectrum $\sigma(\Delta^{(j)}_g)$ of $\Delta^{(j)}_g$ consists of eigenvalues with a finite multiplicity and thus the situation is rather simple. On the other hand, if $M$ is noncompact, then $\sigma(\Delta^{(j)}_g)$ usually contains some continuous part, which cannot be controlled in general, that is, without any further assumptions on $(M,g)$.

A systematic approach to control the \emph{absolutely continuous} part $\sigma_{\mathrm{ac}}(\Delta^{(j)}_g)$ of $\sigma(\Delta^{(j)}_g)$ in the noncompact case is directly motivated by quantum mechanics, namely, the usage of scattering theory. Here the essential idea is as follows: Assume that there is a quasi-isometric metric $\tilde{g}$ on $M$ such that we have some good information about the absolutely continuous part $(\Delta^{(j)}_{\tilde{g}})_{\mathrm{ac}}$ of $\Delta^{(j)}_{\tilde{g}}$. Then once we can show that the wave operators
$W_{\pm}\big(H_{g},H_{\tilde{g}}\big)$ exist and are complete (cf. Theorem \ref{belo} for a precise definition of completeness), they induce unitary equivalences 
$$
(\Delta^{(j)}_{\tilde{g}})_{\mathrm{ac}}\sim (\Delta^{(j)}_{g})_{\mathrm{ac}},\>\text{ in particular, one has }\>\sigma_{\mathrm{ac}}(\Delta^{(j)}_{\tilde{g}})=\sigma_{\mathrm{ac}}(\Delta^{(j)}_g).
$$
Now in order to actually carry through the above program, a typical approach has been to assume that $M$ has a special topological structure and that both metrics $g$, $\tilde{g}$ are in some sense compatible with the latter, e.g.~in the situation of manifolds with cylindrical ends or cusp ends, see in particular \cite{gui} and more recently \cite{huns}. For further references we refer to the extensive literature cited in \cite{HPW}. This approach ultimately leads to the study of direct sums of Sturm-Liouville type operators, which is of course a classical and well-understood field.

A major new development in the scattering approach to spectral geometry has been the paper \cite{MS}, where the authors allow \emph{arbitrary} Riemannian manifolds. There the authors consider Laplacians acting on functions, that is $0$-forms, and their main result can be rephrased as follows (cf. Theorem 0.1 \cite{MS}), where from now on we assume $\dim(M)\geq 2$:

\medskip
\noindent\emph{Assume that $g$, $\tilde{g}$ are complete Riemannian metrics in $M$ with $|\mathrm{Sec}_{g}|, |\mathrm{Sec}_{\gg}|\leq L$ for some $L>0$, such that the covariant $\mathsf{C}^2$-deviation ${}^2|g-\tilde{g}|_g$ of $g$ from $\tilde{g}$ is bounded pointwise from above by a function $\beta: M\to (0,\infty)$ of moderate decay (in particular $g$ and $\tilde{g}$ are quasi-isometric), in a way such that for appropriate constants $a,b,c,C$ one has
$$
\beta^{a}\in\IL^1(M,g),\> \bigl|\beta^b(x){\widetilde{\inj}}_g(x)^{c}\bigr|\leq C\>\text{ for all $x$},
$$
where 
$$
\widetilde{\inj}_g(x):=
\min\bigl\{  \tfrac{\pi}{12 \sqrt{L}} ,\:  \inj_g(x) \bigr\}.
$$
Then the wave operators $W_{\pm}\big(\Delta^{(0)}_{g},\Delta^{(0)}_{\tilde{g}}\big)$ exist and are complete.}
\medskip

On the other hand, this scalar result has been generalized recently in \cite{HPW}, using harmonic radius estimates on the Sobolev scale from \cite{anderson}: There, using a certain decomposition formula (cf. Lemma 3.4 in \cite{HPW}) of the operator
\begin{align}\label{posts}
V^{(0)}=\big(\Delta^{(0)}_{\tilde{g}}+1\big)^{-n} \big(\Delta^{(0)}_{\tilde{g}}  - \Delta^{(0)}_g \big)\big(\Delta^{(0)}_{g}+1\big)^{-n},
\end{align}
the authors prove (cf. Theorem 3.7 in \cite{HPW}) that the assumptions of Belopol\rq{}skii-Birman\rq{}s theorem (cf. Theorem \ref{belo} below) are satisfied under an integrability condition of the form
\begin{align}\label{intpos}
\int_{M} \Id(g,\tilde{g})(x) h^{-(\dim(M)+2)}(x)\mathrm{vol}_g(\Id x)<\infty,
\end{align}
where $\Id(g,\tilde{g}):M\to (0,\infty)$ is a  function which only measures a \emph{zeroth order deviation} of the metrics (and not a \emph{second order} one), and where $h: M\to (0,1]$ is an arbitrary common lower bound on both Sobolev-harmonic radii $r_{g}, r_{\tilde{g}}$. Ultimately, the authors of \cite{HPW} end up with condition (\ref{intpos}), by using generally valid elliptic estimates of the form 
\begin{subequations}
\begin{align}
\bigl|\big(\Delta^{(0)}_{g}+1\big)^{-n} &f(x)\bigr|\leq C\min\{1,r_g(x)\}^{-\dim(M)/2}\|f\|_{\mathsf{L}^2(M,g)},\label{ellest1}\\
\bigl|\Id\big(\Delta^{(0)}_{g}+1\big)^{-n}&f(x)\bigr|_g \leq C\min\{1,r_g(x)\}^{-(\dim(M)/2 +1)}\|f\|_{\mathsf{L}^2(M,g)}, \label{ellest2}
\end{align}
\end{subequations}
where $n$ is large enough, in order to estimate the trace norm of $V^{(0)}$.

\vspace{3ex}

As these are all \emph{scalar results for functions}, the natural question which we address in this paper is: 

\medskip
\noindent\emph{To what extent can one prove a scattering result for the Hodge-Laplacian $\Delta^{(j)}_{*}$ on $j$-forms, which only requires a lower order control on the deviation of the metrics?}
\medskip

To this end, in order to make an effective use of Belopol\rq{}skii-Birman\rq{}s theorem as in \cite{HPW}, we restrict ourselves to the particularly important case of conformal perturbations. Ultimately, the restriction to conformal perturbations turns out to be not restrictive at all for many applications, as e.g. any two sufficiently well-behaved warped product metrics automatically are \lq\lq{}essentially conformally equivalent\rq\rq{} (see the proof of Proposition \ref{warpconfper} below for a precise statement).\\
In order to formulate our main results, we fix a Riemannian metric $g$ on $M$. If $\tilde g$ is another metric on $M$ which is quasi-isometric to $g$, then we denote with
$$
I=I_{g,\tilde g}:\Omega_{\IL^2}(M,g)\longrightarrow\Omega_{\IL^2}(M,\tilde g),\>\omega\longmapsto \omega
$$
the canonical identification operator. Let $\cf:M\to\Reell$ be smooth, so that the conformally equivalent metric $\gg:=\mathrm{e}^{2\cf} g$ is quasi-isometric to $g$, if and only if $\cf$ is bounded.

For any $K>0$ and any function $h:M\to (0,\infty)$, we introduce the following notation: $\IMM_{K,h}(M)$ stands for the space of complete metrics $g\rq{}$ on $M$ with $\min\{1,r'_g\}\geq h$, and with curvature endomorphism bounded from below by $-K$. 

Note that this definition is clearly motivated by the elliptic estimates \eqref{ellest1}, \eqref{ellest2}. Furthermore, given a Borel function $h:M\to (0,\infty)$, the conformal factor $\cf$ will be called an \emph{$h$-scattering perturbation of $g$,} if 
\begin{align}\label{inrt}
\int_{M} \Id(g,\cf)(x) h^{-(\dim(M)+2)}(x)\mathrm{vol}_g(\Id x)<\infty,
\end{align}
where now
\begin{align}\label{firts}
\Id(g,\cf)(x):=\max\bigl\{\sinh(2|\cf(x)|), \left|\Id\cf(x)\right|_{g}\bigr\}, \quad x\in M.
\end{align}
Then with $\Delta=\bigoplus_j\Delta^{(j)}$ the total Hodge-Laplacian, our main result reads as follows:

\begin{introenv}{Theorem}{main}
Let $\cf:M\to \IR$ be smooth with $\cf,|\Id \cf|_g $ bounded, and assume that $g,\gg\in\IMM_{K,h}(M)$ for some pair $(K,h)$, in a way such that $\cf$ is an $h$-scattering perturbation of $g$. Then the wave operators $W_{\pm}(\Delta_{\gg},\Delta_g, I)$ exist and are complete. Moreover, the $W_{\pm}\big(\Delta_{\gg},\Delta_{g}, I\big)$ are partial isometries with initial space $\mathrm{Im}\:P_{\mathrm{ac}}(\Delta_g)$ and final space $\mathrm{Im}\:P_{\mathrm{ac}}(\Delta_{\gg})$.
\end{introenv}

\noindent It is straightforward to check that this theorem applies to the case of arbitrary compactly supported perturbations (see Corollary \ref{cor:4.1}). Morover, combining this Theorem \ref{main} with a result from \cite{bunke} we get the following result, which states that under slightly stronger curvature assumptions, we can drop the conformal equivalence on a compact set:

\begin{introenv}{Corollary}{cor:bunke}
Let $(M,g)$ and $(M,\tilde g)$ be conformal at infinity, i.e.~there are a compact set $K\subset M$ and a smooth function $\cf: M\to \Reell$ such that $\tilde g=\mathrm{e}^{2\cf} g$ on $M\setminus K$. Assume that $\cf$, $|\Id\cf|_g$ are bounded, that $\mathrm{Sec}_g$ is bounded, and that $g,\gg\in\IMM_{L,h}(M)$ for some pair $(L,h)$, in a way such that $\cf$ is a $h$-scattering perturbation of $g$. \\
Then the wave operators 
$W_{\pm}(\Delta_{\tilde g},\Delta_g, I)$ exist and are complete; moreover they are partial isometries with inital space $\mathrm{Im} \: P_{\mathrm{ac}}(\Delta_g)$ and final space $\mathrm{Im} \: P_{\mathrm{ac}}(\Delta_{\tilde g})$.\end{introenv}

Corollary \ref{main2} in Section \ref{sec:3} below states that Theorem \ref{main} also holds in every differential form degree. Moreover when restricted to $0$-forms, it is still more general than the above mentioned Theorem 0.1 from \cite{MS} when applied to the conformal case. This follows from:

\begin{introenv}{Proposition}{ms}
Assume that $\cf:M\to \IR$ is a smooth bounded function, that $g$ is complete such that $|\mathrm{Sec}_{g}|, |\mathrm{Sec}_{\gg}|\leq L$ for some $L>0$, and furthermore that there is a function $\beta$ which is exponentially bounded from below (see Definition \ref{expbnd}), such that the following conditions are satisfied:
\begin{enumerate}[(i)]
\item For some constant $C>0$ one has ${}^1|g-\gg|\leq C\cdot \beta $.
\item There are constants $b\in (0,1)$ with $\beta^{b}\in\IL^1(M,g)$, and $C_1>0$ such that for all $x\in M$,
\begin{equation*}
\widetilde{\inj}_g(x)\ge C_1 \cdot \beta(x)^{\frac{1-b}{\dim(M)+2}}.
\end{equation*}
\end{enumerate}
Then the assumptions of Theorem \ref{main} are satisfied.
\end{introenv}

We also have the following consequence of Theorem \ref{main}:

\begin{introenv}{Corollary}{ffgg} Let $g$ be such that $|\mathrm{Sec}_g|$ is bounded and that $g$ has a positive injecitivity radius (in particular, $g$ is complete). Assume that $\psi:M\to \IR$ is smooth with $\max\{\psi,|\Id\psi|_g,|\mathrm{Hess}_g(\psi)|_g\}$ bounded, and
$$
\int_M  \max\{\sinh(2|\psi(x)|),|\Id\psi(x)|_g\}\mathrm{vol}_g(\Id x) <\infty.
$$
Then the wave operators $W_{\pm}(H_{\gg},H_g, I)$ exist and are complete. 
\end{introenv}

Corollary \ref{ffgg} can be brought into a very applicable form in the case of warped product metrics. Ultimately, as indicated above, we are going to use our scattering results together with results by Antoci \cite{antoci} to control the absolutely continuous $j$-form spectrum for a large class of warped product metrics, from the knowledge of the spectrum of \emph{one special} warped product metric. These facts are included in Section \ref{sec:4}.\\
The reader should notice that in all these results our assumptions on the deviation of the metrics are purely \emph{first order} ones.\\

\medskip
Let us add some remarks on the technical issues of the assumptions, and the proof of Theorem \ref{main}, which also indicate in what sense the case of differential forms is analytically much more involved than the case of functions. 

An effective use of a decomposition formula as (\ref{posts}) which reflects elliptic estimates such as \eqref{ellest2}, requires the underlying operators to be of the form $D^*\!D$. Thus we are led to work with total differential forms and not with forms of a fixed degree, so that we can use the underlying Dirac structure $\Delta_{g}=D_g^*D_g=D^2_g$, where $D_g$ is the Gauss-Bonnet operator. However, $D_g=\Id+\delta_g$ depends itself on $g$, while on functions it is just the differential $\Id$. Ultimately, this is the reason that now we have to require a first order control in the definition (\ref{firts}), which cannot be expected to be dropped. More specifically, in this setting the generalization of the decomposition formula for \eqref{posts} takes the following form:\vspace{2mm}

\begin{introenv}{Proposition}{decompo} Let $g$ be complete and let $\cf,|\Id \cf|_g$ be bounded. Then for $\lambda>0$, $n\geq 1$, the bounded operator
$$
V:=R_{\gg,\lambda}^n (\Delta_{\gg} I -I \Delta_g )R_{g,\lambda}^n:\Omega_{\IL^2}(M,g)\longrightarrow\Omega_{\IL^2}(M,\gg)
$$ 
can be decomposed as
\begin{multline*}
V= R_{\gg,\lambda}^n
\Big(
D_{\gg} \cdot 2\sinh(2\cf) I D_g 
+D_{\gg} I (1-\mathrm{e}^{-2\cf}) \Id -\Id \circ (1-\mathrm{e}^{2\cf}) I D_g \\ 
+D_{\gg}  \inti_{\gg}(\Id\cf)\,\tau\: I 
- \tau \:\inti_g(\Id\cf) D_g
\Big)R_{g,\lambda}^n,
\end{multline*}
where $R_{*,\lambda}:=(\Delta_*+\lambda)^{-1}$ denotes the resolvent, and where $\tau$ is multiplication by a constant in each degree.
\end{introenv}

Indeed, it is essential in the latter result to assume that $|\Id \cf|_g$ is bounded, already to make the right hand side of the formula for $V$ well-defined at all.

Next, we remark that in order to estimate the trace norm of the operator $V$ in terms of the quantitity \eqref{inrt}, the approach from \cite{HPW} would require \emph{first order} estimates as in \eqref{ellest2}, but now for $AR_{g,\lambda}^{n}$, where $A\in \{D_g, \Id, \delta_g\}$. Such estimates seem hard to establish in general. Instead, we take a different approach which relies on the commutator relations $[A, R^{n}_{g,\lambda}]=0$, and which allows us to restrict ourselves to the differential form analogue of the \emph{zeroth order} estimate \eqref{ellest1}. This is the content of:

\begin{introenv}{Proposition}{estim}
Assume that $g\in \IMM_{K,h}(M)$ for some pair $(K,h)$. Then for all sufficiently large $n=n(\dim(M))\in\IN$ there is a $C=C(n,\dim(M))>0$, such that for all sufficiently large $\lambda=\lambda(K,m)>0$ the operator $R^n_{g,\lambda}$ is an integral operator, with a Borel integral kernel 
$$
M\times M\ni (x,y)\longmapsto R^n_{g,\lambda}(x,y)\in \mathrm{Hom}(\Wedge^j \IT^*_y M,\Wedge^j \IT^*_x M)
$$
which satisfies
$$
\int_M\left|R^n_{g,\lambda}(x,y)\right|^2_{\IJJ^2}\vol_g(\Id y)\leq C\cdot h(x)^{-\dim(M)}\text{ for all $x\in M$},
$$
where $|\cdot|_{\IJJ^2}$ stands for the Hilbert-Schmidt norm on the fibers $\mathrm{Hom}(\Wedge^j \ITs_y M,\Wedge^j \ITs_xM)$ (w.r.t.~$g$).
\end{introenv}

\vspace{3ex}
\noindent This paper is organized as follows: In Section \ref{sec:1} we establish some geometric and functional analytic notation, and we provide the reader with some formulae from conformal geometry. In Section \ref{sec:2} we prove and collect some facts on Sobolev harmonic coordinates and the class of metrics $\IMM_{K,h}(M)$. Section \ref{sec:3} is devoted to the proofs of the above Proposition \ref{decompo}, Proposition \ref{estim}, as well as our main result Theorem \ref{main}. Finally,  Section \ref{sec:4} contains the above applications Corollary \ref{cor:bunke},  Proposition \ref{ms}, and Corollary \ref{ffgg}, as well as some explicit applications of Corollary \ref{ffgg} (such as warped product Riemannian manifolds and the above mentioned determination of absolutely continuous $j$-form spectra of warped product metrics).

\vspace{4mm}

\noindent\textbf{Acknowledgements.} The authors would like to thank Jochen Br\"uning for a helpful discussion. We would also like to thank the anonymous referee for very helpful hints that ultimately lead to the formulation of Proposition \ref{warpconfper}. This research has been financially supported by the SFB 647: Raum-Zeit-Materie.

\section{Setting and some facts from conformal Riemannian geometry}\label{sec:1}

Let $M$ be a connected smooth manifold without boundary, with $m:=\dim(M)\geq 2$. The tangent bundle $\IT M$ and all bundles that can be constructed in a smooth functorial way out of it will be considered as complexified, like for example the exterior product $\Wedge^j \ITs M$ and the full exterior bundle $\Wedge\ITs M=\bigoplus^m_{j=0} \Wedge^j \ITs M$, with the usual convention $\Wedge^0 \ITs M:=M\times\IC$. Given smooth complex vector bundles $E_1\to M$, $E_2\to M$, the complex linear space of smooth linear partial differential operators from $E_1$ to $E_2$ of order $\leq k\in\IN_{\geq 0}$ is denoted with $\IDD^{(k)}_{\ICC^{\infty}}(M;E_1,E_2)$, where we write $\IDD^{(k)}_{\ICC^{\infty}}(M;E_1)$ instead of $\IDD^{(k)}_{\ICC^{\infty}}(M;E_1,E_1)$. If nothing else is said, given $P\in \IDD^{(k)}_{\ICC^{\infty}}(M;E_1,E_2)$, $f\in\Gamma_{\IL^1_{\mathrm{loc}}}(M,E_1)$, the expression  $Pf$ is always understood in the distributional sense. For $\alpha\in \Omega^1_{\ICC^{\infty}}(M)$ we denote with 
$$
\exte(\alpha)\in \IDD^{(0)}_{\ICC^{\infty}}\left(M;\Wedge \ITs M\right)
$$
the operator of exterior multiplication with $\alpha$.

\emph{All Riemannian metrics on $M$ are understood to be smooth, and we fix once for all a Riemannian metric $g$ on $M$}. 

The metric is extended canonically to a Hermitian structure on all vector bundles  $E\to M$ that can be constructed in a \lq\lq{}smooth functorial way\rq\rq{} from $\IT M$ (like e.g. $E=\Wedge^j \IT^* M$), and this Hermitian structure will always be denoted by $(\cdot,\cdot)_g$, where then
\begin{align}\label{norm}
|\psi|_{g}:=\scr{\psi}{\psi}_{g}^{1/2}\>\text{ for any section $\psi$ in $E\to M$.}
\end{align}
denotes the corresponding fiber norm. Likewise, the Levi-Civita connection $\nabla_g$ extends to all such bundles to give a Hermitian covariant derivative. In the particular case of $E=\Wedge^j \IT^* M$ we will sometimes indicate the corresponding data by an index \lq\lq{}$j$\rq\rq{}, like e.g. $\nabla_{g,j}$, or $(\cdot,\cdot)_{g,j}$. For example, the Hessian of a smooth function $f:M\to\IC$ becomes $\mathrm{Hess}_g(f)=\nabla_{g,1}\Id f$.\\
We denote with $\z_g$ the Riemannian Borel measure on $M$, and with
$$
\Q_g\in\IDD^{(0)}_{\ICC^{\infty}}(M;\Wedge^2 \IT M)
$$
its curvature endomorphism, and with $\mathrm{Sec}_g$ the sectional curvature.\\
Recall that if $R_g$ stands for the usual Riemannian curvature, then $\Q_g$ is self-adjoint and determined by the equation
$$
\big(\Q_g (X\wedge Y), Z\wedge W\big)_{g}= (R_g(X,Y)W,Z)_g
$$
for all smooth vector fields $W,X,Y,Z$ on $M$ .\\
Moreover, $\inj_{g}(x)\in (0,\infty]$ stands for the $g$-injectivity radius at $x\in M$, $\Id_g(x,y)$ the geodesic distance, and the corresponding open geodesic balls will be denoted with $B_{g}(x,r)$, $r>0$, $x\in M$.\\
We will denote by $\Omega_{\IL^2}(M,g)$ the complex separable Hilbert space space of equivalence classes $\alpha$ of Borel forms on $M$ such that 
\begin{align*}
\left\|\alpha\right\|^2_{g}:=&\int_M|\alpha(x)|^2_{g} \:\z_g(\Id x) <\infty, \\
\text{ with its inner product }\quad\sce{\alpha}{\beta}_{g}=&\int_M \scr{\alpha(x)}{\beta(x)}_{g}  \:\z_g (\Id x),
\end{align*}
with an analogous notation for the Hilbert space of Borel $j$-forms $\Omega^j_{\IL^2}(M,g)$. In view of $\Wedge\ITs M=\bigoplus^m_{j=0}\Wedge^j\ITs M$, we also have $\Omega_{\IL^2}(M,g)=\bigoplus^m_{j=1}\Omega^j_{\IL^2}(M,g)$.\\
For any smooth $1$-form $\alpha$ on $M$, we get the formal adjoint corresponding to exterior muliplication with $\alpha$,
$$
\inti_{g}(\alpha):= \exte(\alpha)^{\dagger_g}\in \IDD^{(0)}_{\ICC^{\infty}}\left(M;\Wedge \ITs M\right),
$$
which is in fact nothing but contraction by the vector field that corresponds to $\alpha$ via $g$. Let us note (recalling the convention (\ref{norm})):

\begin{lemma}\label{inn} For any $\eta\in \Omega^1_{\ICC^{\infty}}(M),\omega\in \Omega_{\ICC^{\infty}}(M)$ one has the pointwise inequality
\begin{align}
|\inti_g(\eta)\omega|_g\leq |\eta|_{g}|\omega|_g,
\end{align}
in particular, as an operator on $\Omega_{\IL^{2}}(M)$, the norm of contraction with a one-form is bounded by
\[
\|\inti_g(\eta) \|_g \le \|\eta\|_{g,\infty}:=\sup_{x\in M}|\eta(x)|_{g}\in [0,\infty].
\]
\end{lemma}

\begin{proof} We omit the dependence on $g$ of several data in the notation. Because contraction is an anti-derivation, the pointwise equality
\begin{align*}
|\inti(\eta)\omega|^2
&=\scr{\exte(\eta)\inti(\eta)\omega}{\omega}
=\scr{(\inti(\eta)\eta) \omega}{\omega}-\scr{\inti(\eta)\exte(\eta) \omega}{\omega} \\
&
=|\eta|^2 |\omega|^2-|\exte(\eta) \omega|^2
\end{align*}
holds. This shows the first statement, and the second statement then follows from
\[
\big|[\inti(\eta)\omega](x)\big|\le \|\eta\|_{\infty} |\omega(x)|.\qedhere
\]
\end{proof}

\vspace{2ex}

We denote by 
\begin{align*}
& \Id^{(j)}\in\IDD^{(1)}_{\ICC^{\infty}}\bigl(M;\Wedge^j \ITs M,\Wedge^{j+1} \ITs M\bigr),\>\>\delta^{(j)}_g\in\IDD^{(1)}_{\ICC^{\infty}}\bigl(M;\Wedge^{j} \ITs M,\Wedge^{j-1} \ITs M\bigr)
\end{align*}
the exterior differential on $j$-forms and, respectively, the formal adjoint of $\Id^{(j-1)}$. Then we can form the Hodge-Laplacian
$$
\Delta^{(j)}_g:=\delta^{(j+1)}_g\Id^{(j)}+\Id^{(j-1)}\delta^{(j)}_g\in\IDD^{(2)}_{\ICC^{\infty}}\bigl(M;\Wedge^{j} \ITs M\bigr),
$$
whose Friedrichs realization in $\Omega^j_{\IL^2}(M,g)$ will be denoted with $H^{(j)}\geq 0$. With 
\begin{align*}
\Id:=\bigoplus^m_{j=0}\Id^{(j)},\>\delta_g:=\bigoplus^m_{j=0}\delta^{(j)}_g\in\IDD^{(1)}_{\ICC^{\infty}}\left(M;\Wedge \ITs M\right)
\end{align*}
we get the underlying Dirac type operator, and respectively the total Hodge Laplacian 
\begin{align*}
D_g:=\Id+\delta_g\in\IDD^{(1)}_{\ICC^{\infty}}\left(M;\Wedge \ITs M\right),\>\>\Delta_g:=D_g^2\in\IDD^{(2)}_{\ICC^{\infty}}\left(M;\Wedge \ITs M\right),
\end{align*}
where the Friedrichs realization of $\Delta_g$ in $\Omega_{\IL^2}(M,g)$ will be denoted with $H_g\geq 0$. In view of
$$
\Delta_g=\bigoplus^m_{j=0} \Delta^{(j)}_g, \>\text{ we also have }\> H_g=\bigoplus^m_{j=0}  H^{(j)}_g\>\text{  as self-adjoint operators}.
$$
If $g$ is (geodesically) complete, then $D_g$, $\Delta_g$ and $\Delta_g^{(j)}$ are essentially self-adjoint on the corresponding space of smooth compactly supported forms \cite{gromov,strich}. For $\lambda>0$, we denote the resolvents with
\begin{align*}
R^{(j)}_{g,\lambda}:=(H^{(j)}_g+\lambda)^{-1}\in\ILL(\Omega^j_{\IL^2}(M,g)),\>\>R_{g,\lambda}:=(H_g+\lambda)^{-1}\in\ILL(\Omega_{\IL^2}(M,g)).
\end{align*}

Finally, let $\mathcal{Q}_g$ denote the sesqui-linear form quadratic form corresponding to $H_g$: It is the closure of the form given by 
$$
(\alpha,\beta)\longmapsto \int_M (D_g\alpha(x), D_g\beta(x))_g \:\mu_g(\Id x), \>\>(\alpha,\beta)\in\Omega_{\ICC^{\infty}_{\mathrm{c}}}(M)\times \Omega_{\ICC^{\infty}_{\mathrm{c}}}(M),
$$
and by functional analytic facts one always has $\dom(\mathcal{Q}_g)=\dom(\sqrt{H_g})$. An observation that will be essential for us in the sequel is that the commutator of $D_g$ and a, say, smooth function $f$ on $M$ is given in terms of the underlying Clifford multiplication, namely,
\begin{equation}\label{diracf}
[D_g,f]=\mathrm{c}_g(\Id f):=\exte(\Id f)-\inti_g(\Id f)\in\IDD^{(0)}_{\ICC^{\infty}}\left(M;\Wedge \ITs M\right),
\end{equation}
which is ultimately equivalent to saying that $D_g$ is of Dirac type \cite{bgv}. Given a smooth function $\cf$ on $M$ we define 
\begin{align*}
\tau&:=\bigoplus^m_{j=0}(m-2j)1_{\Wedge^j \ITs M}\in \IDD^{(0)}_{\ICC^{\infty}}(M;\Wedge \ITs M),\\
\mathrm{e}^{\cf\tau}&:=\bigoplus^m_{j=0}\mathrm{e}^{(m-2j)\cf}1_{\Wedge^j \ITs M}\in \IDD^{(0)}_{\ICC^{\infty}}(M;\Wedge \ITs M).
\end{align*}

If $\tilde g$ a quasi-isometric metric, then we denote with
$$
I=I_{g,\tilde g}:\Omega_{\IL^2}(M,g)\longrightarrow\Omega_{\IL^2}(M, \tilde g),\>\omega\longmapsto \omega
$$
the canonical identification operator. Given a smooth function $\cf:M\to\IR$, we define another metric $\gg:=\mathrm{e}^{2\cf}g$, noting that $g$ and $\gg$ are quasi-isometric, if and only if $\cf$ is bounded. We will frequently use the following results for conformal perturbations:

\begin{proposition}\label{conf} Let $\cf:M\to\IR$ be smooth.\vspace{-1ex}
\begin{enumerate}[a)]
\item One has
\begin{subequations}
\begin{align}\allowdisplaybreaks
\scr{\cdot}{\cdot}_{\gg,j} &= \mathrm{e}^{-2j\cf}\scr{\cdot}{\cdot}_{g,j}\>\text{ for any $j\in \{0,\dots,m\}$,}\label{faser}\\
\mu_{\gg}&=\mathrm{e}^{m\cf}\mu_g\label{mass},\\
\inti_{\gg}(\alpha)&=\mathrm{e}^{-2\cf}\inti_g(\alpha)\>\text{ for any $\alpha\in\Omega^1_{\ICC^{\infty}}(M)$},\\
\nabla_{\gg,X}Y&=\nabla_{g,X}Y+\Id\cf(X) Y+ \Id\cf(Y)X-(X,Y)_g\operatorname{grad}_g(\cf)  \label{levi}\\
&\qquad \text{for all smooth vector fields $X$, $Y$ on $M$} \nonumber\\
\delta_{\gg}&=\mathrm{e}^{-2\cf} (\delta_g -\inti_g(\Id \cf)\tau)\label{deltaeq2},\\
\Delta_{\gg}&= \mathrm{e}^{-2\cf} \bigl(\Delta_g - 2\tau\: \operatorname{Lie}_{\operatorname{grad}_g(\cf)}
+2\inti_g( \Id\cf) \circ \Id   \nonumber\\
&\qquad+4\exte(\Id\cf)\inti_g( \Id\cf)\tau
-2 \exte( \Id\cf)\delta_g\bigr),\\
R_{g_{\psi}}&=\mathrm{e}^{-2\cf} \left(R_g-g \owedge \left(\mathrm{Hess}_g(\cf)- \Id \cf\otimes\Id \cf+\frac{1}{2}|\Id \cf|^2_g\right)  \right),\label{hessil}
\end{align}
where $\owedge$ denotes the Kulkarni-Nomizu tensor product.
\end{subequations}
\item\label{conf:b} If $\cf$ is bounded, then one has 
\begin{equation}\label{iadjoint}
I^*=  \mathrm{e}^{\cf \tau} I^{-1}.
\end{equation}
\item\label{conf:c} Assume that $\cf$ and $|\Id\cf|_g$ are
 bounded. Then one has $I\dom(\mathcal{Q}_g)= \dom(\mathcal{Q}_{\gg})$. 
\end{enumerate}
\end{proposition}

\begin{proof} The proof of part a) is straightforward, the formulas can be found in \cite{besse}, pg.~58f. Part b) then follows easily from (\ref{faser}) and (\ref{mass}).\\
For part c), note that $\dom(\mathcal{Q}_g)$ is the closure of $\Omega_{\ICC^{\infty}_{\mathrm{c}}}(M)$ w.r.t. the Dirac graph norm 
$$
\omega\longmapsto \bigl(\|\omega\|^2_g+\|D_g\omega\|_g^2\bigr)^{1/2}. 
$$
Moreover one has 
$$
\|D_g\omega\|^2_g=\|\Id\omega\|^2_g+\|\delta_g\omega\|^2_g.
$$
Applying \eqref{deltaeq2} and Lemma \ref{inn}, we obtain
\[
\|\delta_\gg\omega\|^2_{\gg}\le \| \mathrm{e}^{-2\cf}\|_{\infty}^2 \|\delta_g\omega\|^2_g+C_m\|\Id \mathrm{e}^{-2\cf}\|_{g,\infty}^2 \|\omega\|^2_g
\]
Writing $g=\mathrm{e}^{-2\cf}\gg$, the same argument shows
\[
\|\delta_g\omega\|^2_g\le \| \mathrm{e}^{+2\cf}\|_{\infty}^2 \|\delta_\gg\omega\|^2_{\gg}+C_m\|\Id \mathrm{e}^{+2\cf}\|_{\gg,\infty}^2 \|\omega\|^2_{\gg}
\]
and therefore that the graphs norms w.r.t $g$ and $\gg$ are equivalent. This proves the claim.
\end{proof}

\section{Harmonic Sobolev coordinates and the class of metrics $\IMM_{K,h}(M)$}\label{sec:2}

In this section, we collect and prove some facts on harmonic coordinates, that will play an essential for our main results. First we recall the classical definition of the Sobolev harmonic radius $r_g(x,p,q)$ from \cite{anderson}. 
\begin{definition} 
Let $p\in (m,\infty)$, $q\in (1,\infty)$, $x\in M$. Then the \emph{ $\mathsf{W}^{1,p}_g$-harmonic radius at $x$ with Euclidean distortion $q$}, $r_g(x,p,q)\in (0,\infty]$, is defined to be the supremum of all $r>0$ such that there is a $\Delta^{(0)}_g$-harmonic chart
$$
\Phi:B_g\big(x,r\big)\longrightarrow U\subset \IR^m
$$
which, with respect to the $\Phi$-coordinates,  satisfies the estimates
\begin{subequations}
\begin{align}
&q^{-1}(\delta_{ij}) \leq g_{ij}\leq q(\delta_{ij})\text{ as symmetric bilinear forms},\label{harm1}\\
&r^{1- \frac{m}{p}}\Bigl(\int_U |\partial_kg_{ij}(y)|^p\Id y\Bigr)^{1/p} \leq q-1\text{ for all $i,j,k\in\{1,\dots,m\}$}.\label{harm2}
\end{align}
\end{subequations}

\end{definition}

The following definitions will be convenient for the formulation of our main results. Recall that $\Q$ stands for the curvature endomorphism.

\begin{definition} 
\begin{enumerate}[a)]
\item For any $K>0$ and any function $h:M\to (0,\infty)$, let 
\begin{align*}
\IMM_{K,h}(M):=&\Big\{ \tilde{g}\>\Big|\> \text{\emph{$\tilde{g}$ is a complete metric on $M$ with $\Q_{\tilde{g}}\geq -K$}}\\
&  \text{  \emph{and $\min\{1,r_g(\cdot,p,q)\}\geq  h$ for some $p\in (m,\infty),q\in (1,\sqrt{2})$}}\Big\}.
\end{align*}

\item Given a Borel function $h:M\to (0,\infty)$ and a smooth function $\cf:M\to \IR$ define
\begin{align*}
\Id(g,\cf)(x)&:=\max\bigl\{\sinh(2|\cf(x)|), \left|\Id\cf(x)\right|_{g}\bigr\}, \quad x\in M,\\
\Id_{h}(g,\cf)&:=\int_M \Id(g,\cf)(x)h(x)^{-(m+2)} \; \mu_{g}(\Id x)\quad \in[0,\infty].
\end{align*}
Then $\cf$ is called a $h$-\emph{scattering perturbation of $g$}, if one has $\Id_h(g,\cf)<\infty$.
\end{enumerate}
\end{definition}

It is not obvious from the definition that $r_g(x,p,q)>0$, but ultimately this follows from classical elliptic PDE theory (cf. \cite{deturck}), or it can also by deduced from from applying Proposition \ref{prop:inj} below near $x$. Furthermore one has the following fact:

\begin{lemma}\label{lem:harm}
For all $p,q$, the capped $\mathsf{W}^{1,p}_g$-harmonic radius $\min\{1,r_g(\cdot,p,q)\}$ is $1$-Lipschitz continuous w.r.t. $g$, that is, for all $x,y\in M$ one has
\begin{equation} |\min\{1,r_g(x,p,q)\}-\min\{1,r_g(y,p,q)\}|\le \Id_g(x,y).\label{lip1}
\end{equation}
\end{lemma}

\begin{proof}
We omit all $g$\rq{}s, fix $p,q$ and set, $r(x):=r(x,p,q)$, $\tilde{r}(x):=\min\{1,r(x)\}$. \\
Let $x\in M$ and let $y\in B(x,\tilde{r}(x))$. This implies that $r(y)\geq \tilde{r}(x)-\Id(x,y).$ Moreover $0<\tilde{r}(x)-\Id(x,y)<1$ because $\tilde{r}(x)=\min\{1,r(x)\}$ and $\Id(x,y)< \tilde{r}(x)$. Therefore we can conclude that 
$$
\min\{1,r(y)\}\geq \min\{r(x),1\}-\Id(x,y)\ \text{that is}\ \tilde{r}(y)\geq \tilde{r}(x)-\Id(x,y).
$$ 
If $\tilde{r}(x)\geq \tilde{r}(y)$ then we can conclude that $|\tilde{r}(x)-\tilde{r}(y)|\leq \Id(x,y)$.  If $\tilde{r}(x)<\tilde{r}(y)$ then $x\in B(y,\tilde{r}(y))$. This implies that $r(x)\geq \tilde{r}(y)-\Id(x,y)$ and this inequality, as before, leads to the conclusion that $\tilde{r}(x)\geq \tilde{r}(y)-\Id(x,y)$ which in turn implies that  $|\tilde{r}(x)-\tilde{r}(y)|\leq \Id(x,y)$.\\Suppose now that $y\notin B(x,\tilde{r}(x))$. If $x\notin B(y,\tilde{r}(y))$ as well then we can conclude immediately that $|\tilde{r}(x)-\tilde{r}(y)|\leq \Id(x,y)$.  If $x\in B(y,\tilde{r}(y))$ then, as above, we have $r(x)\geq \tilde{r}(y)-\Id(x,y)$ that is $r(x)\geq \min\{r(y),1\}-\Id(x,y)$ which in turn implies $\min\{r(x),1\} \geq \min\{r(y),1\}-\Id(x,y)$ that is $\tilde{r}(x)\geq \tilde{r}(y)-\Id(x,y)$. Finally in this last case we have $\tilde{r}(y)>\tilde{r}(x)$ and so we can conclude that $|\tilde{r}(x)-\tilde{r}(y)|\leq \Id(x,y)$.
\end{proof}

In Proposition \ref{prop:inj} below we provide the reader with harmonic radius estimates under lower bounds on the Ricci curvature, that are required for the class $\IMM_{K,h}(M)$. These estimates heavily rely on classical results from \cite{anderson,HPW}. In order to make contact with our main results on scattering below, we add:

\begin{remark} If one has $\Q_g\geq -K$ for some $K>0$, then one automatically has $\mathrm{Ric}_g\geq -K(m-1)$.
\end{remark}

Now we can prove:

\begin{proposition}\label{prop:inj} Assume that
\[
\mathrm{Ric}_g(x)\geq -\frac{1}{\beta^2}\quad \text{and}\quad \inj_g(x)\geq \tilde h(x)\>\text{ for all $x\in M$,}
\]
where $\beta>0$ is a constant and $\tilde h:M\to (0,\infty)$ is a 
continuous function (as $x\mapsto \inj_g(x)$ itself is continuous, such a function always exists).

\begin{enumerate}[a)]
\item If $\tilde h$ is $g$-Lipschitz, then for any $p,q$ there is $C=C(m,p,q)>0$ such that for all $x\in M$ one has 
$$
\min\{r_g(x,p,q),1\}\geq C\min\biggl\{1,\; \frac{\tilde h(x)}{1+\|\Id \tilde h\|_{\infty,g}},\; \beta \biggr\}.
$$
\item If there is a point $x_0\in M$, and constants $c_1>0$, $c_2\geq 0$ such that $\tilde{h}\geq c_1\mathrm{e}^{-c_2\Id_g(\cdot,x_0)}$, then for any $p,q$ there is $C=C(m,p,q)>0$ such that for all $x\in M$ one has 
\[
\min\{r_g(x,p,q),1\}\geq C\min\Bigl\{1,\;\frac{c_1}{\mathrm{e}^{c_2}}\mathrm{e}^{-c_2\Id_g(x,x_0)},\; \beta \Bigr\}.
\]
\end{enumerate}
\end{proposition}

\begin{proof} We will omit the dependence of $g$ in the notation. Assume that the strictly positive continuous function $r_0$ be a lower bound of the homogenized injectivity radius $\iota(x)=\iota_g(x)$ as defined in \cite{HPW,anderson}.
Then a direct consequence of Proposition 2.3 of  \cite{HPW} (which heavily relies on estimates from \cite{anderson}) is that
 there is a $C'=C'(m,p,q)>0$ such that for all $x\in M$ the harmonic radius is bounded from below by
\begin{equation}
r(x,p,q)\geq C'\cdot h(x),\qquad h(x)=\min\bigl\{1,\; r_0(x),\; \beta \bigr\},\label{homgen}\tag{$\ast$}
\end{equation}
so that 
$$
\min\{r(x,p,q),1\}\geq \min\{C',1\}\cdot h(x).
$$

In the cases \emph{a)} and \emph{b)} we can estimate  the homogenized injectivity radius  $\iota(x)$ at $x\in M$ and find an explicit expression for $r_0(x)$.
First we recall the definition of $\iota(x)$.
For any continuous function $f:M\to\IR$ and $t>0$ let  
\[
\myinf{t} f(x):=\inf_{y\in B(x,t)}f(y),
\]
then 
\[
\iota(x):=\sup\{ t>0 \mid \myinf{t} \inj(x)\geq t \}.
\]
Note that $t\mapsto \myinf{t} f(x)$ is non-increasing, and for $t>0$ one has 
$$
\myinf{t} \inj(x)\geq \myinf{t}\,\tilde h(x) 
$$
We will choose $r_0(x)$ such that
\begin{align*}
\iota(x)&=\sup\{ t>0 \mid \myinf{t} \inj(x)\geq t \}
\geq \sup \{ t>0 \mid \myinf{t} \tilde h(x) \geq t \} \ge r_0(x).
\end{align*}

\emph{a)}
Let $L:=\|\Id \tilde h\|_{\infty}$. Then $ \myinf{t}\,\tilde h(x) \geq \tilde h(x)- Lt$
so that 
\begin{align*}
\iota(x)
\geq \sup \{ t>0 \mid \tilde h(x)-Lt\geq t \} = \frac{\tilde h(x)}{1+L}=: r_0(x).
\end{align*}

\emph{b)} Let $b(x):=\Id(x,x_0)$. For $\tilde h(x)=c_1 \mathrm{e}^{-c_2 b(x)}$,
\begin{align*}
\iota(x)&\ge \sup \{ t>0 \mid \myinf{t} \tilde h(x) \geq t \}
= \sup \bigl\{ t>0 \mid \inf_{y\in B(x,t)} c_1 \mathrm{e}^{-c_2 b(y)} \geq t \bigr\} \\
&\ge \sup \bigl\{ t>0 \mid c_1 \mathrm{e}^{-c_2 (b(x)+t)} \geq t \bigr\}
=\sup \bigl\{ t>0 \mid c_1 \mathrm{e}^{-c_2 b(x)} \geq t \mathrm{e}^{c_2 t} \bigr\}
\end{align*}
Because we are only interested in $r_0(x)\le 1$, we conclude further
\begin{align*}
\min\{\iota(x),1\} &\ge \sup \bigl\{ t\in (0,1] \mid c_1 \mathrm{e}^{-c_2 b(x)} \geq t \mathrm{e}^{c_2} \bigr\}
=\frac{c_1}{\mathrm{e}^{c_2}}\mathrm{e}^{-c_2 b(x)}=:r_0(x).
\end{align*}
This completes the proof.\end{proof}

\section{Main results: The existence of the wave operators}\label{sec:3}

This section is completely devoted to the formulation and the proof of our main result Theorem \ref{main} below, which deals with the existence and the completeness of the wave operators $W_{\pm}(H_{\gg},H_g, I)$.\\
The following two propositions are the main technical tools for the proof of Theorem \ref{main}. The first is a decomposition formula for the operator $R_{\gg,\lambda}^n (H_{\gg} I -I H_g )R_{g,\lambda}^n$:

\begin{proposition}\label{decompo} In the situation of Proposition \ref{conf}\ref{conf:c}), let $\lambda>0$, $n\geq 1$ and let $g$ (and thus $\gg$) be complete. Then the bounded operator
$$
R_{\gg,\lambda}^n (H_{\gg} I -I H_g )R_{g,\lambda}^n\longrightarrow \Omega_{\IL^2}(M,g)\>\text{ to }\>\Omega_{\IL^2}(M,\gg)
$$ 
can be decomposed as
\begin{multline}\label{vdec}
R_{\gg,\lambda}^n (H_{\gg} I -I H_g )R_{g,\lambda}^n= \\
R_{\gg,\lambda}^n
\Big(
D_{\gg} \cdot 2\sinh(2\cf) I D_g 
+D_{\gg} I (1-\mathrm{e}^{-2\cf}) \Id -\Id \circ (1-\mathrm{e}^{2\cf}) I D_g \\ 
+D_{\gg}  \inti_{\gg}(\Id\cf)\,\tau\: I 
- \tau \:\inti_g(\Id\cf) D_g
\Big)R_{g,\lambda}^n.
\end{multline}
\end{proposition}

\begin{proof} Let us first note that if $\tilde{g}$ is a complete metric, then 
$$
\dom(\mathcal{Q}_{\tilde{g}})=\big\{\alpha\left|\alpha\in\Omega_{\IL^2}(M,\tilde{g}), D_{\tilde{g}}\alpha\in \Omega_{\IL^2}(M,\tilde{g})\big\}\right. ,
$$
It follows from Proposition \ref{conf}\ref{conf:c}) that
$$
I\dom(\mathcal{Q}_{g})=\dom(\mathcal{Q}_{g_{\psi}}).
$$
We set $g_1:=g$, $g_2:=\gg$, and $R_j:=R_{g_j,\lambda}$. Then, with an obvious notation, let 
$$
V:=R_{2}^n (H_{2} I -I H_1 )R_{1}^n,
$$
and let arbitrary $f_j\in\Omega_{\IL^2}^k(M,g_j), j=1,2$ be given. We further define 
$$
h_j:=R^n_j f_j\in \dom({H_j}^n)\subset \dom(\sqrt{H_j})=\dom(\mathcal{Q}_{j}). 
$$
Then with $\phi:=\mathrm{e}^{\cf}$ we can calculate 
{\allowdisplaybreaks
\begin{align*}
\sce{V f_1}{f_{2}}_{2}&= \sce{R_{2}^n (H_{2} I -I H_1 )R^n_1 f_1}{f_{2}}_{2} \\
&= \sce{D_{2} I h_1}{D_{2} h_{2}}_{2}-\sce{I H_1 h_1}{h_{2}}_{2} \\
\sce{D_{2} I h_1}{D_{2} h_{2}}_{2}
&=\sce{ I (\Id+\phi^{-2}[\delta-\inti_1(\Id\phi/\phi)\tau]) h_1}{D_{2} h_{2}}_{2} \\
&=\sce{ D_{2} I (\Id+\phi^{-2}[\delta-\inti_1(\Id\phi/\phi)\tau]) h_1}{ h_{2}}_{2}\\
&=\sce{ D_{2} I \phi^{-2}(D_1-\inti_1(\Id\phi/\phi)\tau) h_1}{ h_{2}}_{2}
+\sce{ D_{2} I (1-\phi^{-2})\Id  h_1}{ h_{2}}_{2} \\
&=\sce{ D_{2} I \phi^{-2}D_1 h_1}{ h_{2}}_{2}
+\sce{ D_{2} I (1-\phi^{-2})\Id  h_1}{ h_{2}}_{2} \\
&\qquad -\sce{ D_{2} I \phi^{-2}\inti_1(\Id\phi/\phi)\tau h_1}{ h_{2}}_{2}\\
\sce{I H_1 h_1}{h_{2}}_{2} &=\sce{H_1 h_1}{\phi^{m-2k} I^{-1} h_{2} }_1 =\sce{D_1 h_1}{D_1 (\phi^{m-2k} I^{-1} h_{2}) }_1 \\
&\stackrel{\eqref{diracf}}{=}\sce{D_1 h_1}{\phi^{m-2k} D_1 I^{-1} h_{2}}_1+ \sce{D_1 h_1}{(\exte-\inti_1)(\Id \phi^{m-2k}) I^{-1} h_{2} }_1 \\
&=\sce{D_1 h_1}{\phi^{m-2k}  I^{-1} D_1 h_{2}}_1+ \sce{D_1 h_1}{ (\exte-\inti_1)(\Id\phi/\phi) \tau I^* h_{2} }_1 \\
&=\sce{D_1 h_1}{I^*(\phi^2 \Id+\phi^{-2}\delta) h_{2}}_1+ \underbrace{\sce{\tau I (\inti_1-\exte)(\Id\phi/\phi)D_1 h_1}{h_{2} }_{2}}_{=:A} \\
&=\sce{ I D_1 h_1}{(\phi^2 \Id+\delta_{2}-\inti_1(\Id\phi^{-2})\tfrac{\tau}{2}) h_{2}}_{2}+ A\\
&=\sce{ I D_1 h_1}{(\delta_{2}+\phi^2 [\Id-\inti_{2}(\Id\phi^{-2})\tfrac{\tau}{2}]) h_{2}}_{2}+ A,\text{ \small(since $\inti_1 = \phi^2 \inti_{2}$)} \\
&=\sce{  ([\delta_{2}-\tfrac{\tau}{2} \cdot\exte(\Id(\phi^{-2}))]\phi^2\id+ \Id ) I D_1 h_1}{ h_{2}}_{2}+ A \\
&=\sce{  (\delta_{2}\circ\phi^2 \id +\tau \cdot\exte(\Id\phi/\phi)+ \Id ) I D_1 h}{ h_{2}}_{2}+A \\
&=\sce{ D_{2}\circ\phi^2  I D_1 h}{ h_{2}}_{2}
+\sce{  (\Id (1-\phi^2) I D_1 h_1}{ h_{2}}_{2} \\
&\quad +\sce{\tau\cdot \exte(\Id\phi/\phi) I D_1 h_1}{ h_{2}}_{2}
+A \\
&=\sce{ D_{2}\circ\phi^2  I D_1 h_1}{ h_{2}}_{2}
+\sce{  (\Id (1-\phi^2) I D_1 h_1}{ h_{2}}_{2} \\
&\quad+\sce{\tau \cdot\inti_1(\Id\phi/\phi)I D_1 h_1}{h_{2} }_{2}
\end{align*}
}
Altogether we get the decomposition
\begin{multline*}
V=R_{2}^n
\Big(
D_{2} I \phi^{-2}D_1
+D_{2} I (1-\phi^{-2})\Id 
+D_{2} I \phi^{-2}\inti_1(\Id\phi/\phi)\tau \\
-D_{2}\circ\phi^2  I D_1 
- \Id (1-\phi^2) I D_1 
- \tau \inti_1(\Id\phi/\phi)I D_1
\Big)R^n_1,
\end{multline*}
and the claimed formula follows from  $\Id\phi/\phi=\Id\cf$.
\end{proof}

In the sequel the symbol $\IJJ^p$ denotes the $p$-th Schatten class, $p\in [1,\infty]$, of bounded operators acting between two Hilbert spaces (so that $p=1$ is the trace class, $p=2$ is the Hilbert-Schmidt class and $p=\infty$ is the compact class etc.). We will freely use the following well-known facts, valid for all bounded operators $A$, $B$, $C$ whose image and preimage spaces fit together:
\begin{align}\label{schatten}
&\|A\|_{\IJJ^p}=\|A^*\|_{\IJJ^p}, \>\|ABC\|_{\IJJ^p}\leq \|A\|\|B\|_{\IJJ^p}\|C\|\text{ for all $p\in [1,\infty]$}\\
&\|AB\|_{\IJJ^1}\leq \|A\|_{\IJJ^p}\|B\|_{\IJJ^q}\text{ for all $p,q\in (1,\infty)$ with $\tfrac{1}{p}+\tfrac{1}{q}=1$.}
\end{align}

Note that we will apply the above notation fiberwise, as well as in the $\IL^2$-sense.\\
For any smooth vector bundle $E\to M$ let
$$
E^*\boxtimes E=\bigsqcup_{(x,y)\in M\times M}\hspace{-2ex} E^*_y \otimes E_x \longrightarrow M\times M
$$
denote the corresponding (smooth) exterior bundle. One has:

\begin{proposition}\label{estim} Assume that $g\in \IMM_{K,h}(M)$ for some pair $(K,h)$. Then for all $n\in\IN$ with $n\geq m/4+2$ there is a $C=C(m,n)>0$, such that for all 
\begin{align}
\lambda>K \lceil\tfrac{m}{2}\rceil\lfloor\tfrac{m}{2}\rfloor+1=K\max_{j=0,\dots,m}j(m-j) +1,\label{ew}
\end{align}
the operator $R^n_{g,\lambda}$ is an integral operator, with a Borel integral kernel
$$
M\times M\ni (x,y)\longmapsto R^n_{g,\lambda}(x,y)\in \mathrm{Hom}\left(\Wedge \IT^*_y M,\Wedge \IT^*_x M\right)\subset \Wedge\IT^* M\boxtimes\Wedge\IT M
$$
which satisfies
$$
\int_M\left|R^n_{g,\lambda}(x,y)\right|^2_{\IJJ^2}\mu_g(\Id y)\leq Ch(x)^{-m}\text{ for all $x\in M$}.
$$
\end{proposition}
\begin{proof} We will omit the $g$'s in the notation. By the
Bochner-Lichnerowicz-Weitzenb\"ock formula, one has that
\begin{align}\label{pot1}
V_{j}:=\Delta^{(j)}-\nabla^{\dagger}_{j}\nabla_{j}\in\IDD^{(0)}_{\ICC^{\infty}}\bigl(M;\Wedge^j
\ITs M\bigr).
\end{align}
Moreover, the Gallot-Meyer estimate \cite{gallot}  states that under $\Q\geq -K$ one has
\begin{align}
V_{j}\geq -K\cdot j (m-j).\label{pot2}
\end{align}
Then it follows from \eqref{pot1}, \eqref{pot2}
and semigroup domination for covariant Schr\"odinger semigroups (cf.
Theorem 2.13 in \cite{guen}) that
\begin{equation}
\left|\exp\left[-t H^{(j)}\right]\alpha (x)\right|\leq  \exp\left[-t
\big(H^{(0)}-jK (m-j)\big)\right]\left|\alpha \right|(x)\label{semi}
\end{equation}
for all $t\geq 0$, $\alpha\in\Omega^j_{\IL^2}(M)$, where on the rhs,
$x\mapsto \left|\alpha(x) \right|$ is considered as a nonnegative
element of $\IL^2(M)$. But as for any self-adjoint operator $S$ with
$S\geq c$ for some $c\in\IR$, and any $z\in\IR$ with $z<c$ one has
$$
(S-z)^{-n}=\Gamma(n)\int^{\infty}_0 t^{n-1} \mathrm{e}^{tz }\mathrm{e}^{-t
S}\Id t,
$$
we immediately obtain from \eqref{semi} the pointwise inequality
\[
\left|( H^{(j)}+\lambda)^{-n}\alpha (x)\right|\leq  \big(H^{(0)}-jK
(m-j)+\lambda\big)^{-n}\left|\alpha \right|(x).
\]
Next, from a scalar elliptic resolvent estimate in harmonic coordinates
(cf. Theorem B.1 in \cite{HPW}) and the assumption $g\in\IMM_{K,h}(M)$ one
gets a $C_2=C_2(m,n)>0$ such that
\[
( H^{(0)}+1)^{-n}\left|\alpha \right|(x)\leq C_2 
\left\|\alpha\right\|h(x)^{-m},
\]
so putting everything together, keeping (\ref{ew}) in mind, we have for
all $\alpha\in\Omega_{\IL^2}^j(M)$ the estimate
\[
\left|( H^{(j)}+\lambda)^{-n}\alpha (x)\right|\leq  C_2 
\left\|\alpha\right\|h(x)^{-m},
\]
which using
\[
R^n_{\lambda}=( H+\lambda)^{-n}=\bigoplus^m_{j=0}( H^{(j)}+\lambda)^{-n}
\]
implies the existence of a $C_3=C_3(m,n)>0$ such that for all
$\alpha\in\Omega_{\IL^2}(M)$ one has
\begin{align}\label{bound1}
\left|R^n_{\lambda}\alpha (x)\right|\leq  C_2 
\left\|\alpha\right\|h(x)^{-m}.
\end{align}
Now let $\{e_{J}\}_J$ denote a globally defined Borel measurable
$g$-orthonormal frame for $\Wedge\ITs M$ (which of course cannot be chosen
smooth in general, but will not need any further regularity than
measurability). Combining (\ref{bound1}) with Riesz-Fischer\rq{}s
representation theorem for bounded functionals, keeping in mind that $R^n_{\lambda}\alpha$ has a continuous (in fact, a $\mathsf{C}^4$-) representative by the Sobolev embedding theorem, for any $x\in M$ and any
index $J$ we get a unique $R^{n}_{\lambda,x,J}\in \Omega_{\IL^2}(M)$ such
that for all $\alpha\in\Omega_{\IL^2}(M)$ one has 
$$
\int_M\big(R^{n}_{\lambda,x,J}(y),\alpha(y)\big) \mu(\Id y) =
\big(R^n_{\lambda}\alpha (x), e_{J}(x)\big).
$$
Moreover the norm of $R^{n}_{\lambda,x,J}$ is bounded according to $\|R^{n}_{\lambda,x,J}\|\leq C_2$. Defining the
Borel section
\begin{align*}
M\ni x\longmapsto  R^{n}_{\lambda,x}(y) \in  \mathrm{Hom}\left(\Wedge
\IT^*_y M,\Wedge
\IT^*_x M\right),\>R^{n}_{\lambda,x}(y)^*e_{J}(x):=R^{n}_{\lambda,x,J}(y),
\end{align*}
we end up with the formula
$$
R^n_{\lambda}\alpha(x)=\int_MR^{n}_{\lambda,x}(y)\alpha(y)\mu(\Id y).
$$
It remains to show that $ (x,y)\mapsto R^n_{\lambda}(x,y)$ has a jointly
Borel $\mu$-version: To this end, it is sufficient to prove that $(x,y)\mapsto
R^{n}_{\lambda,x,J}(y)$ has a jointly Borel $\mu$- version. Pick a countable ONB
$(\phi_l)_{l\in\IN}$ of $\Omega_{\IL^2}(M)$. In view of
\[
\sce{R^{n}_{\lambda,x,J}}{\phi_l}=
\int_M\big(R^{n}_{\lambda,x,J}(y),\phi_l(y)\big) \mu(\Id y)
=\big(R^n_{\lambda}\phi_l (x), e_{J}(x)\big),
\]
we know that $x\mapsto \left\langle
R^{n}_{\lambda,x,J},\phi_l\right\rangle$ is Borel for all $l$. But now the
$\IL^2$-expansion 
$$
R^{n}_{\lambda,x,J}=\sum_{l\in\IN}\left\langle
R^{n}_{\lambda,x,J},\phi_l\right\rangle  \phi_l
$$
implies that $(x,y)\mapsto R^n_{\lambda}(x,y)$ can be chosen jointly Borel, as the rhs of the latter equation \emph{is} jointly Borel, and the proof is complete.
\end{proof}

Now we can formulate and prove our main result on the existence and completeness of the wave operators $W_{\pm}(H_{\gg},H_g, I)$. We refer the reader to Section \ref{wave} for some corresponding functional analytic notation.

\begin{theorem}\label{main} Let $\cf:M\to \IR$ be smooth with $\cf$, $|\Id\cf|_g$ bounded, and assume that $g,\gg\in\IMM_{K,h}(M)$ for some pair $(K,h)$, in a way such that $\cf$ is a $h$-scattering perturbation of $g$. Then the wave operators 
$$
W_{\pm}(H_{\gg},H_g, I)=\slim_{t\to\pm\infty}\mathrm{e}^{\mathrm{i}tH_{\gg}}I\mathrm{e}^{-\mathrm{i}tH_{g}}P_{\mathrm{ac}}(H_g)
$$
exist and are complete. Moreover, the $W_{\pm}\big(H_{\gg},H_{g}, I\big)$ are partial isometries with inital space $\mathrm{Im}  P_{\mathrm{ac}}(H_g)$ and final space $\mathrm{Im}  P_{\mathrm{ac}}(H_{\gg})$.
\end{theorem}

\begin{proof} Let $g_1:=g$, $g_2:=\gg$. In view of Proposition \ref{conf}\ref{conf:b}), and keeping in mind (\ref{schatten}) and that for all bounded intervals $S\subset \IR$, $\ell\in\IR$, $r>0$, one has
$$
E_{H_j}(S)(H_j+r)^{\ell}=(H_j+r)^{\ell}E_{H_j}(S)\in\ILL(\Omega_{\IL^2}(M,g_j)),
$$
we see that all assumptions of the Belopolskii-Birman theorem (cf. Theorem \ref{belo} below) are satisfied, once we can show that for all $\lambda$ as in Proposition \ref{estim} and all even $n\geq m/2+4$ one has
\begin{align}
\label{hilbert}
&\left\|\big(I^*I-1\big)R^{n}_{1,\lambda}\right\|_{\IJJ^2}<\infty,\\
&\label{spur}
\left\|R_{2,\lambda}^n \big(H_{2} I -I H_1 \big)R_{1,\lambda}^n\right\|_{\IJJ^1}<\infty.
\end{align}

In order to see (\ref{hilbert}), we just have to note $I^*I=\mathrm{e}^{\cf\tau}$, and that by projecting, Proposition \ref{estim} shows that at each degree the operator $R^{(j),n}_{1,\lambda}$ is an integral operator, with a Borel integral kernel
$$
M\times M\ni (x,y)\longmapsto R^{(j),n}_{1,\lambda}(x,y)\in \mathrm{Hom}\bigl(\Wedge^j \ITs_y M,\Wedge^j \ITs_xM\bigr)
$$
that satisfies the same estimate as $R^{n}_{1,\lambda}(x,y)$. It follows that $\bigl(\mathrm{e}^{(m-2j)\cf}-\mathrm{id}\bigr)R^{(j),n}_{1,\lambda}$ is an integral operator, and thus we get
\begin{align*}
&\left\|\big(\mathrm{e}^{\cf\tau}-1\big)R^{n}_{1,\lambda}\right\|_{\IJJ^2}^2\\
&\leq C_1(m)\int_M\int_M \left|(\mathrm{e}^{(m-2j)\cf(x)}-1)\right|^2\bigl|R^{(j),n}_{1,\lambda}(x,y)\bigr|^2_{\IJJ^2}\mu_1(\Id x)\mu_1(\Id y)\\
&\leq C_2(\cf) \int_M h(x)^{-m}\Id(g_1,\cf)(x) \mu_1(\Id x)=C_2(\cf) \Id_{h}(g_1,\cf)<\infty,
\end{align*}
where we have used that $\cf$ is bounded, so that
\[
|\mathrm{e}^{(m-2j)\cf(x)}-1|^2\le \|\mathrm{e}^{(m-2j)\cf}-1\|_\infty |\mathrm{e}^{(m-2j)\cf(x)}-1|
\le C(m,\cf) \sinh(2|\cf(x)|).
\]
It remains to prove \eqref{spur}, which will be shown using the decomposition formula \eqref{vdec}. We only show how to estimate the first summand (noting that in view of Lemma \ref{inn} the other summands can be treated analogously): \\
Let $S(x)=\sinh(2\cf(x))^{1/2}$ be a (complex) square root.
Since under completeness one has $[D_j ,R^n_{j,\lambda}]=0$, we have
\begin{align*}
R^n_{2,\lambda}D_2\sinh(2\cf) I D_1R^n_{1,\lambda}=D_2R^{n/2}_{2,\lambda}R^{n/2}_{2,\lambda}S(x)S(x) I R^{n/2}_{1,\lambda}R^{n/2}_{1,\lambda}D_1
\end{align*}
As $D_k R^n_{k,\lambda}=R^n_{k,\lambda}D_k$ and $I^{-1}$ are bounded, it follows from $SI=I^{-1}S$ that
\begin{align*}
&\bigl\|R^n_{2,\lambda}D_2\sinh(2\cf)  I D_1R^n_{1,\lambda}\bigr\|_{\IJJ^1} \\
&\leq \bigl\|D_2R^{n/2}_{2,\lambda}\bigr\|\bigl\|R^{n/2}_{2,\lambda}S\bigr\|_{\IJJ^2}\bigl\|S I R^{n/2}_{1,\lambda}\bigr\|_{\IJJ^2}\bigl\|R^{n/2}_{1,\lambda}D_1\bigr\|\\
&\leq C\bigl\|R^{n/2}_{2,\lambda}S\bigr\|_{\IJJ^2}\bigl\|S R^{n/2}_{1,\lambda}\bigr\|_{\IJJ^2}
= C\bigl\|\bigl(R^{n/2}_{2,\lambda}S\bigr)^*\bigr\|_{\IJJ^2}\bigl\|S R^{n/2}_{1,\lambda}\bigr\|_{\IJJ^2}\\
&= C\bigl\|S R^{n/2}_{2,\lambda}\bigr\|_{\IJJ^2}\bigl\|S R^{n/2}_{1,\lambda}\bigr\|_{\IJJ^2}.
\end{align*}
However, by Proposition \ref{estim}, the operator $S R^{n/2}_{k,\lambda}$ is an integral operator which satisfies
$$
\bigl\|S R^{n/2}_{k,\lambda}\bigr\|_{\IJJ^2}\leq C(m,n) \sqrt{\Id_{h}(g_1,\cf)}<\infty.
$$ 
This completes the proof.
\end{proof}

\begin{corollary}\label{main2} Under the assumptions of Theorem \ref{main}, let $j\in \{0,\dots,m\}$ and let 
$$
I^{(j)}=I^{(j)}_{g,g_{\cf}}:\Omega^j_{\IL^2}(M,g)\longrightarrow\Omega^j_{\IL^2}(M,\gg),\>\omega\longmapsto \omega
$$
be the canonical identification acting on $j$-forms. Then the wave operators 
$$
W_{\pm}\bigl(H^{(j)}_{\gg},H^{(j)}_g, I^{(j)}\bigr)=\slim_{t\to\pm\infty}\mathrm{e}^{-\mathrm{i}tH^{(j)}_{\gg}}I^{(j)}\mathrm{e}^{\mathrm{i}tH^{(j)}_{g}}P_{\mathrm{ac}}\left(H^{(j)}_g\right)
$$
exist and are complete. Moreover, the $W_{\pm}\big(H^{(j)}_{\gg},H^{(j)}_{g}, I\big)$ are partial isometries with inital space $\mathrm{Im}  P_{\mathrm{ac}}(H^{(j)}_g)$ and final space $\mathrm{Im}  P_{\mathrm{ac}}(H^{(j)}_{\gg})$.
\end{corollary}

\begin{proof} This follows from Theorem \ref{main}, noting that one has
$$
H_g=\bigoplus^m_{j=0}H^{(j)}_g,\>H_\gg=\bigoplus^m_{j=0}H^{(j)}_\gg
$$
so that we get the corresponding orthogonal decompositions of the spectral measures
$$
E_{H_g}=\bigoplus^m_{j=0}E_{H^{(j)}_g},\>E_{H_\gg}=\bigoplus^m_{j=0}E_{H^{(j)}_\gg},
$$
so that $I=\bigoplus^m_{j=0}I^{(j)}$ completes the proof.
\end{proof}

\section{Applications and examples}\label{sec:4}

This section is devoted to some abstract applications of Theorem \ref{main} which are then applied to some explicit examples. 

\subsection{Applications} 

Firstly, we can handle the prototypical case of metric perturbations, namely topological perturbations:

\begin{corollary}\label{cor:4.1} Assume that $g$ is complete with $\Q_g\geq -K$ for some $K>0$ and that $\tilde{g}$ is a complete metric on $M$ which is conformally equivalent to $g$ and which coincides with $g$ at infinity. Then the assumptions of Theorem \ref{main} (and thus of Corollary \ref{main2}) are satisfied.
\end{corollary}

\begin{proof} Let $\tilde{g}=\mathrm{e}^{2\cf} g$. By assumption, $\cf$ is compactly supported, so that $\cf$ and $|\Id\cf|_{g}$ are bounded, and in view of \eqref{levi}, $\Q_\gg$ is bounded from below. As the harmonic radius depends continuously on $x$ by Lemma \ref{lem:harm}, it follows that all assumptions of Theorem \ref{main} are satisfied with
$$
h(x):=\min\{1,r_g(x,p,q),r_\gg(x,p,q)\}\text{ for all $p>m$, $1<q<\sqrt{2}$, }
$$
noting that
$$
\Id_h(g,\cf)\leq  C_{\cf}\int_{\supp(\Id(g,\cf))}\hspace{-2ex}h(x)^{-m-2} \mu_g(\Id x) <\infty,
$$
where we have used again that $\cf$ is compactly supported.
\end{proof}

Combining with Theorem 4.1 in \cite{bunke} for the Gauss-Bonnet operator $D_g=\Id+\delta_g$, we obtain 
the following variant of Theorem \ref{main}, which essentially states that under slightly stronger curvature assumptions, we can drop the conformal equivalence on compact subsets:

\begin{corollary}\label{cor:bunke}
Let $(M,g)$ and $(M,\tilde g)$ be conformal at infinity, i.e.~there are a compact set $K\subset M$ and a smooth function $\cf: M\to \Reell$ such that $\tilde g=\mathrm{e}^{2\cf} g$ on $M\setminus K$. Assume further that $\cf$, $|\Id\cf|_g$ are bounded, that $\mathrm{Sec}_g$ is bounded, and that $g,\gg\in\IMM_{L,h}(M)$ for some pair $(L,h)$, in a way such that $\cf$ is a $h$-scattering perturbation of $g$. Then the wave operators 
$
W_{\pm}(H_{\tilde g},H_g, I)$
exist and are complete; moreover they are partial isometries with inital space $\mathrm{Im} \: P_{\mathrm{ac}}(H_g)$ and final space $\mathrm{Im} \: P_{\mathrm{ac}}(H_{\tilde g})$.
\end{corollary}

\begin{proof}
From Theorem \ref{main} we know that $W_{\pm}(H_{\gg},H_g, I_{g,\gg})$ exist and are complete, and Theorem 4.1 of \cite{bunke} shows that $W_{\pm}(H_{\tilde g},H_\gg,  I_{\gg,\tilde g})$ exist and are complete.

The chain rule for wave operators
\[
W_{\pm}(H_{\tilde g},H_g, I_{\gg,\tilde g}\circ I_{g,\gg})=W_{\pm}(H_{\tilde g},H_\gg,  I_{\gg,\tilde g})W_{\pm}(H_{\gg},H_g, I_{g,\gg})
\]
implies that $W_{\pm}(H_{\tilde g},H_g, I)$, where by definition
$$
I:=I_{g,\tilde g}=I_{\gg,\tilde g}\circ I_{g,\gg}
$$
exists. Using Proposition XI.5(c) from \cite{reed}, we get that $W_{\pm}(H_{\tilde g},H_g,  I)$ is complete from the existence of
\[
W_{\pm}(H_g,H_{\tilde g}, I^{-1})=W_{\pm}(H_g,H_{\gg}, I^{-1}_{g,\gg})W_{\pm}(H_\gg, H_{\tilde g}, I^{-1}_{\gg,\tilde g}).
\]
As in the proof of \eqref{hilbert} we conclude that $(I^{*} I-1) E_{H_g}(S)$ is compact for any bounded interval $S$. Then using Proposition 5(d) and Lemma 2 from \cite[chapter XI.3]{reed}, we get the statement about the partial isometries.
\end{proof}

As a more sophisticated application, we show that the assumptions of our main result are weaker (in the conformal case) than those of the main result of \cite{MS}, where only functions, that is $0$-forms, are treated. To this end, we add the simple:

\begin{definition}\label{expbnd} We say that a continuous decreasing function $\beta: [0,\infty)\to (0,\infty)$, with $\beta<1$ in the complement of compact set, is \emph{exponentially bounded from below}, if there are $C_1>0$, $C_2\geq 0$ such that 
$$
\beta(r)\geq C_1\mathrm{e}^{-C_2 r}\text{ for all $r\geq 0$.}
$$ 
\end{definition}

Now we can prove:

\begin{proposition}\label{ms} Assume that $\cf:M\to \IR$ is a smooth bounded function, that $g$ is complete such that sectional curvatures are bounded $|\mathrm{Sec}_{g}|, |\mathrm{Sec}_{\gg}|\leq L$ for some $L>0$, and furthermore that there is a function $\beta$ which is exponentially bounded from below, and a point $x_0\in M$ such that with $\beta(x):=\beta(1+\Id_g(x,x_0))$ the following conditions are satisfied:
\begin{enumerate}[(i)]
\item For some constant $C>0$ one has 
\begin{equation}
{}^1|g-\gg|:=|g-\gg|_g+|\nabla_g-\nabla_\gg|_g \leq C\cdot \beta \label{normdiffbound}
\end{equation}
\item\label{zwo} There are constants $b\in (0,1)$ with $\beta^{b}\in\IL^1(M,g)$, and $C_1>0$ such that for all $x\in M$,
\begin{equation}
\widetilde{\inj}_g(x):=
\min\bigl\{  \tfrac{\pi}{12 \sqrt{L}} ,\:  \inj_g(x) \bigr\}
\ge C_1 \cdot \beta(x)^{\frac{1-b}{m+2}}.\label{lowinjbound}
\end{equation}
\end{enumerate}
Then the assumptions of Theorem \ref{main} are satisfied, i.e.~one has $|\Id\cf|_g\in \IL^{\infty}(M)$, and $g,\gg\in\IMM_{K,h}(M)$ for some pair $(K,h)$, in a way such that $\cf$ is a $h$-scattering perturbation of $g$.
\end{proposition}

\begin{proof} Let us first check that $g,\gg\in \IMM_{K,h}$ for some $K>0$ and an appropriate $h:M\to (0,\infty)$.\\
Clearly, by the curvature assumption, both curvature endomorphisms are bounded from below by a constant.
To construct $h$, choose $0<\eta\le 1$ with
\[
\eta g \le \gg \le \eta^{-1} g,
\]
then  Proposition D.1 from \cite{HPW} together with \eqref{lowinjbound} implies that one has
\[
\inj_\gg(x) \geq \min\Bigl\{ \tfrac{\eta^2 \pi}{12\sqrt{L}} \: ,\:
C_1\tfrac{\eta}{2}\beta(x)^{\frac{1-b}{m+2}}   \Bigr\}=:\tilde h(x).
\]
Thus we can use Proposition \ref{prop:inj} to  conclude that for all $p,q$ one has
\begin{equation}\label{hbound1}
\min\{1,r_g(x,p,q),r_\gg(x,p,q)\}\geq h(x):=\min\bigl\{ c_1 \mathrm{e}^{-c_2 \frac{1-b}{m+2} \Id(x,x_0)} ,\:c_3\bigr\}
\end{equation}
with $c_1,c_3$ depending on $\eta,L, p, q$. Finally it remains to show that $\Id_{h}(g,\cf)<\infty$. 
To that end we will first show
\begin{equation}
\Id_{h}(g,\cf)(x)\leq C_3 \cdot {}^1|g-\gg|(x)\le \tilde C\beta(x).\label{einsnorm}
\end{equation}
Clearly $|g-\gg|_g=|\mathrm{e}^{2\cf}-1| $, so that 
\[
\sinh(2 |\cf(x)|)\le C_4 |g-\gg|(x).
\]
Furthermore
 recall from \eqref{levi} that for any smooth vector field $Y$ on $M$ one has
$$ 
(\nabla_\gg-\nabla_g)(Y)=\Id\cf(\cdot)\otimes Y+\Id\cf(Y)(\cdot) -g(\cdot,Y) \text{grad}_g(\cf) \quad 
$$
Let $\{X_i\}$ be a smooth local orthonormal frame of vector fields w.r.t. $g$. Then
\begin{align*}
|\Id\cf|_g^2&=\sum_\ell |X_\ell(\cf)|^2, \\
|\nabla_\gg-\nabla_g|_g^2&=\sum_{j,k=1}^m |(\nabla_\gg-\nabla_g)(X_j,X_k)|^2 \\
&=\sum_{j,k}^m | X_j(\cf) X_k +X_k(\cf) X_j-\delta_{jk} \sum_\ell X_\ell(\cf) X_\ell|^2\\
&=\sum_{j<k} 2 | X_j(\cf) X_k +X_k(\cf) X_j|^2 + \sum_\ell |X_\ell(\cf)|^2,\\
\text{i.e.} \qquad\qquad
|\Id\cf|_g &\le |\nabla_\gg-\nabla_g|_g.
\end{align*}
Together with \eqref{hbound1} this shows \eqref{einsnorm} (and also that $|\Id\cf|_g$ is bounded).\\
Now the proof of $\Id_{h}(g,\cf)<\infty$ is almost the same as in  Remark 3.9 from \cite{HPW}; for the convenience of the reader we repeat the short argument. \\
We decompose $M=M_1\sqcup M_2$, where $M_1=\{x\in M \mid h(x)=c_3\}$ with $c_3$ from \eqref{hbound1}, and $M_2=M\setminus M_1$.
On $M_2$ we know 
\[
h^{-(m+2)}(x)=c_5 \mathrm{e}^{c_2 \frac{1-b}{m+2} \Id_g(x,q)}\le c_6 \beta(x)^{- \frac{1-b}{m+2}},\quad x\in M_2
\]
and from \eqref{einsnorm}, \eqref{normdiffbound}
\begin{align*}
\Id_{h}(g,\cf) 
&=\int_M \Id(g,\cf)(x)h(x)^{-(m+2)} \; \mu_{g}(\Id x) \\
&\le C \int_{M_2} \beta(x)^{b } \; \mu_{g}(\Id x)
+\hat C \int_{M_1} \beta(x) \; \mu_{g}(\Id x) <\infty
\end{align*}
where we have used \eqref{zwo} and that $\beta<1$ outside a compact set.
\end{proof}

\begin{remark} Note that although it deals with differential forms, the assumptions of Proposition \ref{ms} are still weaker than the ones from the main result Theorem 0.1 from \cite{MS} which only deals with functions (which however treats not necessarily conformal perturbations!) Firstly, and this is the main point of our results, we only have to assume a \emph{first order condition} on the deviations of the metrics, whereas their assumption $g\sim^2_{\beta} \gg$ (cf. Definition 1.9 from \cite{MS}) is a \emph{second order} one of the form 
$$
{}^2|g-\gg|(x):=|g-\gg|_g(x)+\sum^1_{i=0} |\nabla^i_g (\nabla_g-\nabla_\gg)|_g (x)\leq C\cdot\beta(x).
$$
Secondly, we can allow a larger class of \lq\lq{}control functions\rq\rq{} $\beta$. Indeed, Theorem 0.1 from \cite{MS} requires the function $\beta$ to be of \lq\lq{}a moderate decay\rq\rq{} (cf. Definition 1.4 in \cite{MS}), which is a stronger assumption than ours on $\beta$. 
\end{remark}

\subsection{Examples}

Finally, let us come to some explicit examples which satisfy the assumptions of the above results. Firstly, general manifolds with bounded geometry can be treated as follows:

\begin{corollary}\label{ffgg} Let $g$ be such that $|\mathrm{Sec}_g|$ is bounded and that $g$ has a positive injecitivity radius (in particular, $g$ is complete). Assume that $\psi:M\to \IR$ is smooth with $\max\{\psi,|\Id\psi|_g,|\mathrm{Hess}_g(\psi)|_g\}$ bounded, and
$$
\int_M  \max\{\sinh(2|\psi(x)|),|\Id\psi(x)|_g\}\mu_g(\Id x) <\infty.
$$
Then the wave operators $W_{\pm}(H_{\gg},H_g, I)$ exist and are complete. 
\end{corollary}

\begin{proof} In view of Theorem \ref{main}, it is sufficient to prove that $g,\gg\in \IMM_{K,h}$ for some \emph{constants} $K>0,h>0$. Clearly this is the case for $g$. On the other hand, $g_{\psi}$ is complete, and as $|\mathrm{Hess}_g(\psi)|_g$ and $|\Id\psi|_g$ are bounded it follows from the perturbation formula (\ref{hessil}) that $|\mathrm{Sec}_{g_{\psi}}|$ is bounded. Now as both metrics are complete and quasi-isometric and both with bounded curvature,  $ \mathrm{inj}_g(M)>0$ automatically implies  $\mathrm{inj}_{g_{\psi}}(M)>0$ (see e.g. Proposition D.1 from \cite{HPW}), which completes the proof of $g,\gg\in \IMM_{K,h}$ for some constants $K>0,h>0$. 
\end{proof}

More specifically, Corollary \ref{ffgg} can be brought into the following convenient form in the particularly important case where the \lq\lq{}known\rq\rq{} Riemannian manifold $(M,g)$ has a {\em warped product structure}:

To this end let $M$ be a smooth connected manifold (without boundary) $\dim(M)=n+1$, let $U\subset M$ be a smooth compact submanifold with boundary and $\dim(N)=n$. Let us label by $N:=\partial U$ the boundary of $U$ and by $U'$ the interior of $U$. Assume that there exists a smooth diffeomorphism: $F:M\setminus U'\rightarrow [1,\infty)\times N$. Finally consider a smooth metric $g$ on $M$ such that $(F^{-1})^*(g|_{M\setminus U'})=h^2\Id r^2+f^2g_N$ where $f:[1,\infty)\rightarrow [0,\infty)$ and $h:[1,\infty)\rightarrow [0,\infty)$ are smooth and $g_N$ is a smooth metric on $N$.  

\begin{proposition}
\label{condwar}
Let $M$, $N$, $U$, $U'$, g and $F:M\setminus U'\rightarrow [1,\infty)\times N$ be as above. Then the following assertions hold: 
\begin{itemize}
\item If $\inf h>0$, then $g$ is complete.
\item If $\inf\min\{h,f\}>0$, then there exists $\epsilon>0$ such that 
$$
\inf_{x\in M } \mu_g(B_g(x,\epsilon))>0.
$$
\item Assume that $h=1$ and that $(\log(f))''$, $ (\log(f)')^2$ and $1/f^{2}$ are bounded functions on $[1,\infty)$. Then $(M,g)$ is complete with bounded sectional curvatures  and positive injectivity radius.
\end{itemize}
\end{proposition}

\begin{proof}
Let $\tilde{x}:M\rightarrow (0,\infty)$ be a smooth function such that $\tilde{x}|_{M\setminus U'}=x\circ F$. Then $\tilde{x}$ is a proper function with bounded gradient and the completeness statement follows using Gordon's completeness criterion \cite{gordon1,gordon2} (which states that completeness is equivalent to the existence of a smooth proper function with bounded gradient). The second statement follows immediately observing that $U$ is compact and that on $[1,\infty)\times N$ (which through $F$ is isometric to $M\setminus U'$) we have 
$$
\mu_{g'}(\Id p,\Id r)=h(r)f^n(r)\Id r \mu_{g_N}(\Id p)\>\text{ where $g':=(F^{-1})^*(g|_{M\setminus U'})$}.
$$
 Finally we deal with the third statement. According to the formulas for the curvature of warped product metrics proved in \cite{pli} page 393, the assumptions on $f$ guarantee that the sectional curvatures of $g$ are bounded. This property, together with the second property, implies that $\mathrm{inj}_g(M)>0$ which follows from a classical result of Cheeger-Gromov-Taylor \cite{CGT}, page 47. The proof is thus completed.
\end{proof}

We point out that $f(r):=r^{b}$, with $b\geq 0$, satisfies all the assumptions of the last statement of Proposition \ref{condwar}.\\
 Here the case $b=0$ (with $h=1$) corresponds to Riemannian manifolds \emph{with cylindrical ends} (these kind of metrics belong to the family of Melrose's $\mathbf{b}$-metrics which are intensively studied in \cite{cyl}). \\
The case $b=1$ (with $h=1$) corresponds to so called Riemannian manifolds with \emph{conical ends}. This kind of metric is  a particular case of a more general class of metrics called \emph{scattering metrics} which are intensively studied in \cite{con}.  \\
Another function that satisfies all the assumptions of the last statement of Proposition \ref{condwar} is given by  $f(r)=\mathrm{e}^{r}$. This kind of metric (with $h=1$) is isometric to a conformally compact metric defined on the interior of $U$ and the latter kind of metrics are studied for instance in \cite{com}.

\begin{corollary}
\label{warconf} Assume in the above warped product situation  that $|\mathrm{Sec}_g|$ is bounded, that $g$ has a positive injectivity radius  and that $\beta:[1,\infty)\to (0,\infty)$ is a bounded  Borel function with $\beta\in \IL^1([1,\infty),h(r)f^{n}(r)\Id r)$. Then for any bounded smooth function $\psi:M\to\IR$ with bounded $g$-Hessian and 
$$
\max\{ \sinh(|2\psi|),|\Id\psi|_g\}|_{ F^{-1} (r,q)}\leq  \beta(r)\>\text{ for all $(r,q)\in [1,\infty)\times N$,}
$$
the wave operators $W_{\pm}(H_{\gg},H_g, I)$ exist and are complete.
\end{corollary}

%\begin{proof} We are going to use Corollary \ref{ffgg}. Note first that $|\Id\psi|_g$ is bounded, as $\beta$ is so. Now given a compact subset $Z\subset M$ let us denote by $\Id_{g,Z}:M\rightarrow [1,\infty)$ the distance function to $Z$. Then it is easy to check that for every $p\in M\setminus U$ we have $\Id_{g,N}(p)\geq r$ where $F(p)=(r,q)\in [1,\infty)\times N$. In particular we have $\Id_{g,N}(p)=r$ when $h\equiv 1$. 

%Indeed, for $p_0\in N$, $p\in M\setminus U$ we have $ r\leq   \Id_g(p,p_0)$, $\beta( \Id_g(p_0,p))\leq \beta(r)$ and finally

\begin{proof} We are going to use Corollary \ref{ffgg}. Note first that $|\Id\psi|_g$ is bounded, as $\beta$ is so. By Corollary \ref{ffgg} and the compactness of $U$ we only have to check that 
$$
\int_{M\setminus U\rq{}}\max\{\sinh(2|\psi(x)|),|\Id\psi(x)|_g\}\mu_g(\Id x)<\infty,
$$ 
but clearly
\begin{align*}
&\int_{M\setminus U\rq{}}\max\{\sinh(2|\psi(x)|),|\Id\psi(x)|_g\}\mu_g(\Id x)\\
&\leq \int_N\int^{\infty}_{1}\beta(r) h(r)f^n(r)\Id r\Id \mu_{g_N}( q)\\
&\leq \mu_{g_N}(N)\int^{\infty}_{1}\beta(r) h(r)f^n(r)\Id r<\infty,
\end{align*}
so that all assumptions of Corollary \ref{ffgg} are satisfied by $(g,\psi)$.
\end{proof}

Let us specify the condition on the \lq\lq{}control function\rq\rq{} $\beta$ from Corollary \ref{warconf} for the above mentioned three special cases of warped products: \\
If $g$ has a cylindrical end, then we simply have to require that $\beta\in \IL^1([1,\infty),\Id r)$.\\
If $g$ has a conical end, then our  condition on $\beta$ becomes $\beta\in \IL^1([1,\infty),r^n\Id r)$.\\ 
Finally, in case $h=1$ and $f(r)=\mathrm{e}^{r}$, our condition on $\beta$ becomes $\beta\in \IL^1([1,\infty),\mathrm{e}^{nr}\Id r)$.\vspace{2mm}

We can now formulate a more sophisticated application of our main result, which is motivated by observing that any two sufficiently well-behaved warped product metrics are \lq\lq{}essentially conformally equivalent\rq\rq{} (see the proof of Proposition \ref{warpconfper} below for precise statements).

\begin{proposition}
\label{warpconfper}
Let $M$, $N$, $U$, $U'$ and $F:M\setminus U'\rightarrow [1,\infty)\times N$ be as in Proposition \ref{condwar}. Let $g$ and $\tilde{g}$ be two warped product metrics such that 
\begin{itemize}
\item one has
$$
(F^{-1})^*(g|_{M\setminus U'})=\Id r^2+r^{2b}g_N\>\text{ for some $b$ with $0\leq b \leq 1$} 
$$
\item one has
$$
(F^{-1})^*(\tilde{g}|_{M\setminus U'})=\Id r^2+f^2g_N\>\text{ with $f:[1,\infty) \rightarrow [1,\infty)$ smooth}
$$
\item the function $\max\{f^{-2}, (\log(f))'',(\log(f)')^2\}$ is bounded
\item with $\phi=\phi_{f,b}:[1,\infty)\to [0,\infty)$ defined by
$$
\phi(r):=\begin{cases}
&(1-b)^{\frac{1}{1-b}}\left(\int_1^r\frac{1}{f(s)} \Id s\right)^{\frac{1}{1-b}},\>\text{ if $0\leq b<1$}\\
&\mathrm{e}^{\int_1^r\frac{1}{f(s)}\Id s},\>\text{ if $b=+1$}
\end{cases}
$$
assume that 
\begin{equation}
\label{hessian}
\left|\frac{\phi'''\phi'-(\phi'')^2}{(\phi')^2}\right|\>\>\text{ is bounded}
\end{equation}
 and that 
\begin{equation}
\label{secondbound}
\max\left\{ \sinh\Big(|\log\big((\phi')^2\big)|\Big),\left|\frac{\phi''}{\phi'}\right|\right\}\leq  \beta
\end{equation}
for some bounded $\beta\in \IL^1([1,\infty), f^{n}\Id r)$.
\end{itemize}
Then the wave operators $W_{\pm}(H_g,H_{\tilde{g}}, I)$ exist and are complete.
\end{proposition}

\begin{proof} First of all we observe that, by Proposition \ref{condwar}, $\tilde{g}$ and $g$ are complete, with bounded sectional curvature and with positive injectivity radius. Let $\chi$ be a smooth and positive function on $M$ such that $\mathrm{e}^{2\chi}|_{M\setminus U'}$ coincides with $(\phi')^2\circ F$. Then, by \eqref{hessian} and \eqref{secondbound}, we have that also $g_{\chi}$ is complete with bounded sectional curvatures and positive injectivity radius and by Corollary \ref{warconf} we can conclude that  the wave operators $W_{\pm}(H_{g_{\chi}},H_g, I)$ exist and are complete. Now consider the metric $g$ and $g_{\chi}$. They are isometric on $M\setminus U'$ through an isometry that on $[1,\infty)\times N$ is given by $(s,p)=(\phi(r),p)$.  Therefore we can use Theorem 4.1 of \cite{bunke} to conclude  that $W_{\pm}(H_ g,H_{g_{\chi}},  I_{g, g_{\chi}})$ exist and are complete. Finally, using the same argument used at the end of the proof of Corollary \ref{cor:bunke}, we can finally conclude that the wave operators  $W_{\pm}(H_g,H_{\tilde{g}}, I)$ exist and are complete.
\end{proof}

Finally, we show how the latter results can be used to actually control the absolutely continuous spectrum of a certain metric from the knowledge of the absolutely continuous spectrum of another metric, by showing that the wave operators exist and are complete. Let $\sigma_{\mathrm{ac}}(T)$, resp. $\sigma_{\mathrm{ess}}(T)$, denote the absolutely continuous, resp. the essential spectrum of a self-adjoint operator $T$. We recall that $H^{(j)}_g$ denotes the Friedrichs realization of the Hodge-Laplacian given by the metric, acting on $j$-forms.

\begin{corollary}
\label{computations}
Let $g$ and $\tilde{g}$ be as in Proposition \ref{warpconfper}. \\
a) If $b=0$  then for every $j=0,...,n+1$ we have
\begin{equation}
\label{firstcomp}
\sigma_{\mathrm{ac}}\big(H^{(j)}_{\tilde{g}}\big) \subset \bigcup_{k\in\IN}\Big([\lambda_k^{(j)},\infty)\cup[\lambda_k^{(j-1)},\infty)\Big)
\end{equation}
with $(\lambda_k^{(j)})_{k\in\IN}$  (resp. $(\lambda_k^{(j-1)})_{k\in\IN}$) the eigenvalues of the Hodge-Laplacian  $H^{(j)}_{g_N}$ (resp. $H^{(j-1)}_{g_N}$) acting on $N$.\\ 
b) Assume now that $U'$ is diffeomorphic to the open Euclidean ball $B(0,1)\subset \mathbb{R}^{n+1}$. \\
If $b=0$ then for every $j=0,...,n+1$ we have
\begin{equation}
\label{firstcompII}
\sigma_{ac}\big(H^{(j)}_{\tilde{g}}\big) = [\overline{\lambda^{(j)}},\infty),
\end{equation}
  where $\overline{\lambda^{(j)}}:=\min\{\lambda_0^{(j)},\lambda^{(j-1)}_0\}$ is the minimum of the lowest eigenvalue $\lambda_0^{(j)}$ of $H^{(j)}_{g_{\mathbb{S}^n}}$ and the lowest eigenvalue $\lambda^{(j-1)}_0$ of  $H^{(j-1)}_{g_{\mathbb{S}^n}}$, with $g_{\mathbb{S}^n}$ the standard metric on the unit sphere $\mathbb{S}^n$.\\
Finally, if $0<b\leq1$ then for every $j=0,...,n+1$ we have 
\begin{equation}
\label{firstcompIII}
\sigma_{\mathrm{ess}}\big(H^{(j)}_{\tilde{g}}\big)=\sigma_{\mathrm{ac}}\big(H^{(j)}_{\tilde{g}}\big) = [0,\infty).
\end{equation}
\end{corollary}

\begin{proof}
By Proposition \ref{warpconfper} we know that $\sigma_{\mathrm{ac}}\big(H^{(j)}_{\tilde{g}}\big)=\sigma_{ac}\big(H^{(j)}_{g}\big)$. Consider a diffeomorphism of $M$ that, on $[1,\infty)\times N$, is given by $(t,p)=(\mathrm{e}^r,p)$. Through this diffeomorphism we get a metric $g'$, isometric to $g$, such that on $[1,\infty)\times N$ the metric $g\rq{}$ takes the form $\mathrm{e}^{2t}\Id t^2+\mathrm{e}^{2tb}g_N$. Now the assertions follow directly from the results stated on page 251 in \cite{antoci1} and on page 1750-1751 on \cite{antoci}, where the author explicitely calculates the essential spectrum of $g\rq{}$ in the general case, and the absolutely continuous spectrum of $g\rq{}$ in case $U\rq{}$ is the Euclidean ball. For part a) we additionally use that the essential spectrum of a self-adjoint operator always contains its absolutely continuous spectrum.
\end{proof}

\appendix 
\section{Belopol\rq{}skii-Birman theorem}\label{wave}

For the convenience of the reader we cite a variant of the Belopolskii-Birman theorem, which is precisely Theorem XI.13 in \cite{reed}:

\begin{theorem}\label{belo}\emph{(Belopol\rq{}skii-Birman)} For $k=1,2$, let $H_k$ be a self-adjoint operator in a Hilbert space $\IHH_k$, where $E_{H_k}$ denotes the operator valued spectral measure, $\mathcal{Q}_k$ the sesqui-linear form, and $P_{\mathrm{ac}}(H_k)$ the projection onto the absolutely continuous subspace of $\IHH_k$ corresponding to $H_k$. Assume that $I:\IHH_1\to \IHH_2$ is bounded operator which satisfies
\begin{itemize}
\item $I$ has a two-sided bounded inverse
\item For any bounded interval $S\subset \IR$ one has 
\begin{align*}
E_{H_2}(S)(H_2I-IH_1)E_{H_1}(S) &\in \IJJ^1(\IHH_1, \IHH_2),\\
(I^*I-1)E_{H_1}(S) &\in \IJJ^{\infty}(\IHH_1)
\end{align*}
\item either $I\dom(Q_1)=\dom(Q_2)$, or $I\dom(H_1)=\dom(H_2)$. \end{itemize}
Then the wave operators 
$$
W_{\pm}(H_{2},H_1, I)=\slim_{t\to\pm\infty}\mathrm{e}^{\mathrm{i}tH_{2}}I\mathrm{e}^{-\mathrm{i}tH_{1}}P_{\mathrm{ac}}(H_1)
$$
exist and are complete, where completeness means that 
$$
\left(\mathrm{Ker} \: W_{\pm}(H_{2},H_1, I)\right)^{\perp}=\mathrm{Im}\:  P_{\mathrm{ac}}(H_1), \quad\overline{\mathrm{Im} \: W_{\pm}(H_{2},H_1, I)}=\mathrm{Im}\:  P_{\mathrm{ac}}(H_2).
$$
Moreover, $W_{\pm}\big(H_{2},H_1, I\big)$ are partial isometries with inital space $\mathrm{Im} \: P_{\mathrm{ac}}(H_1)$ and final space $\mathrm{Im} \: P_{\mathrm{ac}}(H_2)$.
\end{theorem}

\bibliographystyle{alpha}
\bibliography{scattering-refs}

\end{document}